\documentclass[10pt]{amsart}
\usepackage[margin=1.4in]{geometry}
\usepackage{colortbl}
\newtheorem{theorem}{Theorem}[section]
\newtheorem{lemma}[theorem]{Lemma}
\newtheorem{proposition}{Proposition}[section]

\theoremstyle{definition}
\newtheorem{definition}[theorem]{Definition}
\newtheorem{remark}[theorem]{Remark}

\newcommand{\R}{\mathbb R}

\newcommand{\T}{\mathbb T}
\newcommand{\N}{\mathbb N}

\newcommand{\bq}{\begin{equation}}
\newcommand{\eq}{\end{equation}}

\newcommand{\e}{\varepsilon}
\newcommand{\lt}{\left}
\newcommand{\rt}{\right}
\newcommand{\mc}{\mathcal{C}}

\newcommand{\pa}{\partial}

\newcommand{\intr}{\int_{\R^3}}

\newcommand{\mv}{\mathcal{V}}

\numberwithin{equation}{section}

\begin{document}

\title[Global existence of weak solutions for Navier-Stokes-BGK system]{Global existence of weak solutions for Navier-Stokes-BGK system}

\author[Choi]{Young-Pil Choi}
\address[Young-Pil Choi]{\newline Department of Mathematics and Institute of Applied Mathematics\newline
Inha University, Incheon 402-751, Republic of Korea}
\email{ypchoi@inha.ac.kr}

\author[Yun]{ SEOK-BAE YUN }
\address[Seok-Bae Yun]{\newline Department of Mathematics\newline Sungkyunkwan University, Suwon 440-746, Republic of Korea}
\email{sbyun01@skku.edu}




\begin{abstract} In this paper, we study the global well-posedness of a coupled system of kinetic and fluid equations. More precisely, we establish the global existence of weak solutions for Navier-Stokes-BGK system consisting of the BGK model of Boltzmann equation and incompressible Navier-Stokes equations coupled through a drag forcing term. This is achieved by combining weak compactness of the particle interaction operator based on Dunford-Pettis theorem, strong compactness of macroscopic fields of the kinetic part relied on velocity averaging lemma and a high order moment estimate, and strong compactness of the fluid part by Aubin-Lions lemma. 

\end{abstract}

\maketitle \centerline{\date}

\section{Introduction}

In this paper, we address the existence of weak solutions for a particle-fluid system in which the BGK model of Boltzmann equation and the incompressible Navier-Stokes equations are coupled through a drag force:
\begin{align}\label{main_sys}
\begin{aligned}
&\hspace{0.2cm}\pa_t f + v \cdot \nabla_x f + \nabla_v \cdot ((u-v)f) = \mathcal{M}(f) - f ,\cr
&\pa_t u + u\cdot\nabla_x u + \nabla_x p - \mu\Delta_x u= - \int_{\R^3} (u-v) f\,dv,\cr
&\hspace{2.6cm}\nabla_x\cdot u = 0,\cr
\end{aligned}
\end{align}
subject to initial data
\[
\lt(f(x,v,0), u(x,0)\rt) =: (f_0(x,v), u_0(x)). 
\]
Here $f(x,v,t)$ denotes the number density function on the phase point $(x,v) \in \T^3 \times \R^3$ at time
$t\in \R_+$, and $u(x,t)$ and $p(x,t)$ are the fluid velocity and the hydrostatic pressure on $x\in \mathbb{T}^3$ at time $t\in \R_+$, respectively. $\mu$ is the kinematic viscosity of the fluid.
The local Maxwellian $\mathcal{M}(f)$ is defined by
\[
\mathcal{M}(f)(x,v,t) = \frac{\rho_f(x,t)}{\sqrt{(2\pi T_f(x,t))^{3}}}\exp\lt(- \frac{|v - U_f(x,t)|^2}{2T_f(x,t)}\rt),
\]
where the macroscopic fields of $f$: $\rho_f, U_f$, and $T_f$ are defined by
\begin{align*}
\rho_f(x,t) &:= \int_{\R^3} f(x,v,t)\,dv, \cr
\rho_f(x,t) U_f(x,t) &:= \int_{\R^3} vf(x,v,t)\,dv,\cr
3\rho_f(x,t)T_f(x,t) &:= \int_{\R^3} |v-U_f(x,t)|^2 f(x,v,t)\,dv,
\end{align*}
respectively. These relations give the following cancellation properties:
\begin{align*}
\int_{\mathbb{R}^3}\big\{\mathcal{M}(f)-f\big\}\left(1, v, |v|^2\right)dv=0.
\end{align*}
Note that this provides the conservation of mass, momentum and energy for the BGK model. However, in our coupled model (\ref{main_sys}), this
only leads to conservation of mass. 

The most general model to describe the dynamics of rarefied particles suspended in a fluid is the Navier-Stokes-Boltzmann system coupled through the drag force term. Due to various technical difficulties, however, the global-in-time existence of solutions for such model is currently not available.  In this paper, we consider the case in which the  interactions between the particles are described by  the nonlinear relaxation operator of the BGK model. This is meaningful in the following two senses.

First, the BGK model is one of the most widely used model equation of the Boltzmann equation in physics and engineering. This is due to the qualitatively reliable results produced by the BGK model at much lower computational cost compared to that of the Boltzmann equation.


Secondly, even though existence theories for particle-fluid systems are well studied nowadays, most of the results dealing with the interactions between the suspended particles consider the linear interaction operators. To the best knowledge of the authors, our result seems to be the first result to consider the particle-fluid model with a nonlinear collision operator for particle interactions.\newline


{\bf History 1: Navier-Stokes-Vlasov system:}
Recently, the study on particle-fluid system is gathering a lot of attentions due to their applications, for example, in the study of sedimentation phenomena, fuel injector in engines, and compressibility of droplets of the spray, etc \cite{BBJM, BDM03, Rourke, RM, VASG, Will}. Along with that applicative interest, the mathematical analysis for various modelling is also emphasized. 
In the case when  the direct particle-particle interactions are absent, there are a number of literature on the global existence of solutions; weak solutions for Vlasov or Vlasov-Fokker-Planck equation coupled with homogeneous/inhomogeneous fluids are studied in \cite{CCK14, CK15, Ham98, MV, WY14, Yu13}, strong solutions near a global Maxwellian for Vlasov-Fokker-Planck equation coupled with incompressible/compressible Euler system are obtained in \cite{CDM11, CKL, DL}. We also refer to \cite{Choi16, Choi17} for the large-time behavior of solutions and finite-time blow-up phenomena in kinetic-fluid systems. Despite those fruitful developments on the existence theory, to the best knowledge of the authors, global existence of solutions for kinetic-fluid models where collisions between the particles are taken into account has not been studied so far. It is worth mentioning that the local-in-time smooth solutions for the Vlasov-Boltzmann/compressible Euler equations are studied in \cite{Math10} and the global existence of weak solutions of Vlasov/incompressible Navier-Stokes equations with a linear particle interaction operator taking care of the breakup phenomena is established in \cite{BDM14, YY}. In \cite{BCHK12, CL16}, Vlasov/Navier-Stokes system with a nonlinear particle interaction operator describing an asymptotic velocity alignment behavior is considered and the global existence of weak solutions is obtained.\newline

{\bf History 2: BGK model:}  In spite of its important role as a fundamental model connecting the particle level description and the fluid level description of gaseous systems, the applications of the Boltzmann equation at the physical or engineering level have often been limited by the high numerical cost involved in the numerical computations of the collision operator. This is especially so if one is interested in dealing with specific flow problems.
Looking for a model equation that shares important features of the Boltzmann equation, and therefore, successfully mimics the dynamics of the Boltzmann equation, Bhatnagar et al, and independently Walender, introduced a relaxation model of the Boltzmann equation, which is called the BGK model. Since then, the BGK model has seen a wide range of applications in engineering and physics due to its reliable results at much lower computational cost compared to that of the Boltzmann equation.

The mathematical study of the BGK model can be traced back to \cite{P} where Perthame established the existence of weak solutions. Perthame and Pulvirenti later studied the existence of unique mild solution in a weighted $L^{\infty}$ space \cite{PP}. These works were fruitfully extended to several directions: gases in the presence of external forces \cite{WZ}, plasma \cite{Z-VP1,Z-VP2}, solutions in $L^p$ spaces \cite{ZH}, ellipsoidal extension \cite{Y} and gas mixture problems \cite{KP}. The existence of classical solutions and its exponential stabilization near equilibrium are studied in \cite{Y JMP,Y SIAM}.
The results on the stationary problems in a slab can be found in \cite{BY,Ukai BGK}. BGK model also saw various applications in the study of various macroscopic limits \cite{DMOS,LT,Mellet,MMM,SR,SR1,SR2}. We omit the survey on
the numerical computations related to the BGK model, interested readers may refer to \cite{CP,DP,FS,M,PiPu,RSY}.

\subsection{Main result} Before we define our solution concept and state the main result, we  define
norms, function spaces and notational conventions.
\begin{itemize}
\item We  denote by $C$ a generic, not necessarily identical, positive constant.
\item For functions $f(x,v), g(x)$, $\|f\|_{L^p}$ and $\|g\|_{L^p}$ denote the usual $L^p(\T^3 \times \R^3)$-norm and $L^p(\T^3)$-norm, respectively.
\item $\|f\|_{L_q^\infty}$ represents a weighted $L^{\infty}$-norm:
\[
\|f\|_{L_q^\infty} := ess\sup_{x,v}(1+|v|^q)f(x,v).
\]
\item For any nonnegative integer $s$, $H^s$ denotes the $s$-th order $L^2$ Sobolev space.
\item $\mc^s([0,T];E)$ is the set of $s$-times continuously differentiable functions from an interval $[0,T]\subset \R$ into a Banach space $E$, and $L^p(0,T;E)$ is the set of the $L^p$ functions from an interval $(0,T)$ to a Banach space $E$.
\end{itemize}
In order to state our main theorem on the global existence of weak solutions to the system \eqref{main_sys}, we also introduce functions spaces as follows:
\[
\mathcal{H}:= \{w \in L^2(\T^3) : \nabla_x \cdot w = 0 \} \quad \mbox{and} \quad \mathcal{V}:=\{w \in H^1(\T^3) : \nabla_x \cdot w =0\}.
\]
We then define a notion of weak solutions to the system \eqref{main_sys}.
\begin{definition}\label{def_weak}We say that $(f,u)$ is a weak solution to the system \eqref{main_sys} if the following conditions are satisfied:
\begin{itemize}
\item[(i)] $f \in L^\infty(0,T;(L^1_+ \cap L^\infty)(\T^3 \times \R^3))$,
\item[(ii)] $u \in L^\infty(0,T;\mathcal{H}) \cap L^2(0,T;\mathcal{V}) \cap \mc^0([0,T];\mathcal{V}')$,
\item[(iii)] for all $\phi \in \mc^1_c(\T^3 \times \R^3 \times [0,T])$ with $\phi(x,v,T) = 0$,
$$\begin{aligned}
&-\int_{\T^3 \times \R^3} f_0\phi_0\,dxdv - \int_0^T\int_{\T^3 \times \R^3} f\lt( \pa_t \phi + v \cdot \nabla_x \phi + (u-v)\cdot \nabla_v \phi\rt)dxdvdt\cr
&\hspace{3cm}= \int_0^T \int_{\T^3 \times \R^3} \lt(\mathcal{M}(f) - f\rt)\phi\,dxdvdt,
\end{aligned}$$
\item[(iv)] for all $\psi \in \mc^1_c(\T^3 \times [0,T])$ with $\nabla_x \cdot \psi = 0$ for almost all $t$,
$$\begin{aligned}
&-\int_{\T^3} u_0\cdot \psi_0\,dx + \int_{\T^3} u\cdot \psi\,dx - \int_0^T\int_{\T^3} u \cdot \pa_t \psi\,dxdt + \int_0^T\int_{\T^3} (u\cdot\nabla_x)u\cdot\psi\,dxdt\cr
&\qquad = -\int_0^T\int_{\T^3} \nabla_x u : \nabla_x \psi\,dxdt - \int_0^T\int_{\T^3 \times \R^3} f(u-v)\cdot\psi\,dxdvdt.
\end{aligned}$$
\end{itemize}
\end{definition}
We are now ready to state our main result:
\begin{theorem}\label{thm_main} Let $T>0$. Suppose that the initial data $(f_0, u_0)$ satisfy
\[
f_0 \in L^\infty(\T^3 \times \R^3), \quad \int_{\R^3} |v|^2 f_0(x,v)\,dv \in L^\infty(\T^3), \quad \mbox{and} \quad u_0 \in L^2(\T^3).
\]
Then there exists at least one weak solution to the system \eqref{main_sys} in the sense of Definition \ref{def_weak} satisfying the following estimates:
$$\begin{aligned}
&(i)\,\,\,\,\,\, \|f\|_{L^\infty(\T^3 \times \R^3 \times (0,T))} \leq C\|f_0\|_{L^\infty(\T^3 \times \R^3)},\cr
&(ii)\,\,\,\, \frac12\lt(\int_{\T^3 \times \R^3} |v|^2 f\,dxdv + \int_{\T^3} |u|^2\,dx\rt) + \mu\int_0^t\int_{\T^3} |\nabla_x u|^2\,dxds + \int_0^t \int_{\T^3 \times \R^3} |u-v|^2f\,dxdvds\cr
&\qquad \qquad \qquad \leq C\lt(\int_{\T^3 \times \R^3} |v|^2 f_0\,dxdv + \int_{\T^3} |u_0|^2\,dx\rt),\cr
&(iii)\,\,\int_{\T^3 \times \R^3} f|\ln f|\,dxdv+\int^T_0\!\!\int_{\T^3 \times \R^3}
\big\{\mathcal{M}(f)-f\big\}\ln f \,dxdvdt
\leq C_{f_0,T},
\end{aligned}$$
for almost every $t \in (0,T)$.
\end{theorem}

One of the key elements in the proof is the derivation of the  third moment estimate that remains uniformly bounded with respect to the mollification parameter $\varepsilon$. To derive the weak compactness of the local Maxwellian, we first need to obtain the compactness of the macroscopic fields. For the compactness of the local density and bulk velocity, the second moment estimate combined with the velocity averaging lemma is enough to derive the desired result. However, we need a moment estimate strictly higher than 2 to derive the compactness of the local temperature(see \cite{P}).
In view of this, we observe that the third moment of the regularized distribution function $f_{\varepsilon}$ can be controlled by the kinetic energy of the suspended particles and a fluid-particle
type estimate:
\[
\int^T_0\!\!\!\int_{\T^3 \times \R^3} f_\e |v|^3\,dxdvdt\leq C\lt(\|(\eta\star u_{\varepsilon}-v)f_\e(1+|v|)\|_{L^1}+\int_0^T\int_{\T^3 \times \R^3} f_\e|v|^2\,dxdvdt\rt),
\]
for some $C>0$ independent of $\e$, which in turn is bounded by $L^5$ norm of the fluid velocity.

For the existence of solutions to the fluid equations, a strong compactness is required to control the convection term. For this, we again need to have some uniform bounds for the local density  and local moments together with the total energy estimates. This, combined with the smoothing effect from the viscosity enables us to use the Aubin-Lions lemma to have the strong compactness.\newline

The outline of this paper is as follows: In Section 2, we record several technical lemmas. In Section 3, we set up a regularized approximate system for the Navier-Stokes-BGK model \eqref{main_sys}.
Then, we prove the existence of the regularized model in Section 4, and derive several key a priori estimates independent of the regularizing parameter in Section 5. Section 6 is devoted
to the proof of the main theorem.

%
%

\section{Preliminaries: Auxiliary Lemmas}\label{sec_pre}
In this section, we record various technical lemmas that will be crucially used later.
%
%
We first state the lower bound estimate of the local temperature, which is essential for the local Maxwellian to be well-defined.
\begin{lemma}\label{moments}\emph{\cite{PP}} There exists a positive constant $C_q$, which depends only on $q$, such that
\begin{eqnarray*}
\rho_f(x,t) \leq C_q\|f\|_{L^{\infty}_q}T^{3/2}_f(x,t) \quad (q>3 \mbox{ or }q=0).
\end{eqnarray*}
\end{lemma}
We also need to control the growth of the local Maxwellian by that of the distribution functions:
\begin{lemma}\label{M<f}\emph{\cite{PP}} Suppose $\|f\|_{L^{\infty}_q}<\infty$ for $q>5$. Then
there exists a positive constant $C_q$, which depends only on $q$, such that
\[
\|\mathcal{M}(f)\|_{L^{\infty}_q}\leq C_q \|f\|_{L^{\infty}_q} \quad (q > 5 \mbox{ or } q=0).
\]
\end{lemma}
The next lemma says that, unlike the above estimate, the constant depends also on the final time and the lower bounds of macroscopic fields if we are to control the growth of derivatives either.
\begin{lemma}\label{h}\emph{\cite{Y}}
Assume that $f$ satisfies
\begin{enumerate}
\item $\|f\|_{L^{\infty}_q}+\|\nabla_{x,v}f\|_{L^{\infty}_q}<C_{1}$,
\item $\rho_{f}+|U_f|+ T_f<C_{2}$,
\item $\rho_{f}, T_{f}>C_{3}$,
\end{enumerate}
for some constants $C_{i}>0$ $(i=1,2,3)$.
Then, we have
\begin{align*}
\|\mathcal{M}(f)\|_{L^{\infty}_q}+\|\nabla_{x,v}\mathcal{M}(f)\|_{L^{\infty}_q}\leq
C_{T}\left\{\|f\|_{L^{\infty}_q}+\|\nabla_{x,v}f\|_{L^{\infty}_q}\right\},
\end{align*}
where $C_T > 0$ depends only on $C_{1}$, $C_{2}$, $C_{3}$ and the final time $T$.
\end{lemma}
The Lipschitz continuity of the local Maxwellian can be measured in the same weighted $L^{\infty}_q$ space as follows:
\begin{lemma}\label{h}\emph{\cite{Y}}
Assume $f,g$ satisfy ($h$ denotes either $f$ or $g$)
\begin{enumerate}
\item $\|h\|_{L^{\infty}_q}<C_{1}$,
\item $\rho_{h}+|U_h|+ T_h<C_{2}$,
\item $\rho_{h}, T_{h}>C_{3}$,
\end{enumerate}
for some constants $C_{i}>0$ $(i=1,2,3)$. Then, we have
\begin{align*}
\|\mathcal{M}(f)-\mathcal{M}(g)\|_{L^{\infty}_q}\leq C_{T}\|f-g\|_{L^{\infty}_q},
\end{align*}
where $C_T>0$ depends only on $C_{1}$, $C_{2}$, $C_{3}$ and the final time $T$.
\end{lemma}

In the lemma below, we give an interpolation-type inequality for local moments of $f$. For this, we set
$(k=0,1,2,\cdots)$
\[
m_k f (x,t) := \int_{\R^3} |v|^k f (x,v,t)\,dv \quad \mbox{and} \quad M_k f (t) := \int_{\T^3 \times \R^3} |v|^k f(x,v,t)\,dxdv.
\]

\begin{lemma}\label{lem_mom} \emph{\cite{BDGM}} Let $\beta > 0$ and $g \in L^\infty_+(\T^3 \times \R^3 \times (0,T))$ with $m_\beta g(x,t) < \infty$ for almost every $(x,t)$. Then we have
\[
m_\alpha g(x,t) \leq \lt(\frac{4\pi}{3}\|g(t)\|_{L^\infty} + 1\rt)m_\beta g(x,t)^{\frac{\alpha + 3}{\beta + 3}} \quad \mbox{a.e. } (x,t),
\]
for any $\alpha < \beta$.
\end{lemma}

We next state the velocity averaging lemma.
\begin{lemma}\label{lem_veloa}\emph{\cite{CCK14}}
For $1 \leq p < 5/4$, let $\{g^n\}_n$ be bounded in $L^p(\T^3 \times \R^3 \times (0,T))$. Suppose that $f^n$ is bounded in $L^\infty(0,T;(L^1 \cap L^\infty)(\T^3 \times \R^3))$ and $|v|^2 f^n$ is bounded in $L^\infty(0,T;L^1(\T^3 \times \R^3))$. If $f^n$ and $g^n$ satisfy the equation
\[
\partial_tf^n + v \cdot \nabla_x f^n = \nabla_v^k g^n, \quad f^n|_{t=0} = f_0 \in L^p(\T^3 \times \R^3),
\]
for a multi-index $k$. Then, for any $\psi(v)$, such that $|\psi(v)| \leq c|v|$ as $|v| \to \infty$, the sequence
\[
\lt\{ \int_{\R^3} f^n \psi(v)\,dv\rt\}_n
\]
is relatively compact in $L^p(\T^3 \times (0,T))$.
\end{lemma}
%
%

\section{Global existence for a regularized system}\label{sec_gl_reg}
In this section, we consider a regularized system of \eqref{main_sys}. As in \cite{BDGM}, we regularize the fluid velocity in the drag forcing and convection terms, and apply a high-velocity cut-off to the drag force in the fluid part to relax some difficulties in the system \eqref{main_sys}. More precisely, let $\e>0$ and $\eta$ be a standard mollifier:
\[
0 \leq \eta \in \mc^\infty_0(\T^3), \quad \mbox{supp}_x\eta \subseteq B(0,1), \quad \int_{\T^3} \eta(x)\,dx = 1,
\]
and we set a sequence of smooth mollifiers $\eta_\e(x) = (1/\e^3)\eta(x/\e)$. We also introduce a cut-off function $\gamma_\e \in \mc^\infty(\R^3)$:
\[
\mbox{supp} \gamma_\e \subseteq B(0,1/\e), \quad 0 \leq \gamma_\e \leq 1, \quad \gamma_\e = 1 \mbox{ on } B(0,1/(2\e)), \quad \mbox{and} \quad \gamma_\e \to 1 \mbox{ as } \e \to 0.
\]
Then the regularized system for the system \eqref{main_sys} is defined as follows:
\begin{align}\label{sys_reg}
\begin{split}
&\qquad\pa_t f_\e + v \cdot \nabla_x f_\e + \nabla_v \cdot ((\eta_\e \star u_\e-v)f_\e) = \mathcal{M}(f_\e) - f_\e,\cr
&\pa_t u_\e + (\eta_\e \star u_\e)\cdot\nabla_x u_\e + \nabla_x p_\e - \mu\Delta_x u_\e= - \int_{\R^3} \gamma_\e(v)(u_\e-v) f_\e \,dv,\cr
&\hspace{4cm}\nabla_x\cdot u_\e = 0,
\end{split}
\end{align}
subject to regularized initial data:
\[
\lt(f_\e(x,v,0), u_\e(x,0)\rt) =: (f_{0,\e}(x,v), u_{0,\e}(x)), \quad (x,v) \in \T^3 \times \R^3,
\]
Here $\star$ represents the convolution with respect to the spatial variable $x$.
$u_{0,\e}$ is any  $\mc^\infty$ approximation of $u_0$ such that
$ u_{0,\e} \to u_0$ strongly in $L^2(\T^3)$ as $\e \to 0$, and
$f_{0,\e}$ is defined by
\[
f_{0,\e}= \eta\star \big\{f_01_{f_0<1/\e}\big\} +\e e^{-|v|^2}.
\]
where $1_A$ denotes the characteristic function on $A$. Note that $f_{0,\e}$ satisfies $f_{0,\e} \to f_0$ strongly in $L^p(\T^3 \times \R^3)$ for all $p < \infty$ and weakly-$*$ in $L^\infty(\T^3 \times \R^3)$, $ M_2 f_{0,\e} \to M_2 f_0$ strongly in $L^\infty(\T^3)$ and uniformly bounded with respect to $\e$,

In the following two sections, we prove the proposition below on the global-in-time existence of weak solutions and local-in-time uniform bound estimates of the regularized system \eqref{sys_reg}.
%
%
%
%
\begin{proposition}\label{prop_ext_reg} $(1)$ For any $T>0$ and $\e >0$, there exists at least one weak solution $(f_\e,u_\e)$ of the regularized system \eqref{sys_reg} in the sense of Definition \ref{def_weak}.\newline
\noindent$(2)$  Moreover, there exists a $T_* \in (0,T]$, which only depends on $T$, $\|u_0\|_{L^2} + M_2 f_0$, and $\|f_0\|_{L^\infty}$ such that
\begin{itemize}
\item Total energy estimate:
\bq\label{uni_te}
\sup_{0 \leq t \leq T_*}\lt(\|u_\e(t)\|_{L^2}^2 + M_2 f_\e(t) + \int_0^t \|\nabla_x u_\e(s)\|_{L^2}\,ds \rt) \leq C_1.
\eq
\item Fluid-kinetic mixed estimate:
\bq\label{est_fk}
\|(\eta_\e \star u_\e - v)(1+|v|)f_\e\|_{L^1}<C(f_0,u_0, T_*).
\eq
\item Third moment and entropy estimate:
\[
\|M_3 f_\e\|_{L^1(0,T_*)} + \int_{\mathbb{T}^3\times\mathbb{R}^3} f_\e(t)|\ln f_\e(t)|\,dxdv\leq C(f_0, u_0,T_*).
\]
\end{itemize}
Here, in particular, $C_1 >0$ depends only on $T_*$, $T$, $\|u_0\|_{L^2} + M_2 f_0$, and $\|f_0\|_{L^\infty}$.
\end{proposition}
Since the proof is rather long, we divide the proof into two parts in Section \ref{sec_pro1} (Existence and Uniqueness) and Section \ref{sec_pro2} (Uniform-in-$\e$ estimates )  below.
%
%
%
%
\section{Proof of Proposition \ref{prop_ext_reg} (1): Existence of $(f_\e, u_\e)$}\label{sec_pro1}
We construct the solution $(f_\e, u_\e)$ to the regularized system (\ref{sys_reg}) as a limit of  the approximation sequence $(f^{n}_\e, u^{n}_\e)$ for the system \eqref{sys_reg} given by the following decoupled and linearized system:
\begin{align}\label{sys_rd}
\begin{aligned}
&\hspace{1cm}\pa_t f^{n+1}_\e + v \cdot \nabla_x f^{n+1}_\e + \nabla_v \cdot ((\eta_\e \star u^n_\e-v)f^{n+1}_\e) = \mathcal{M}(f^n_\e) - f^{n+1}_\e,\cr
&\pa_t u^{n+1}_\e + (\eta_\e \star u^{n+1}_\e)\cdot\nabla_x u^{n+1}_\e + \nabla_x p^{n+1}_\e - \mu\Delta_x u^{n+1}_\e= - \int_{\R^3} \gamma_\e(v)(u^n_\e-v) f^n_\e\,dv,\cr
&\hspace{4.5cm}\nabla_x\cdot u^{n+1}_\e = 0,
\end{aligned}
\end{align}
with the initial data and first iteration step:
\[
\lt(f^n_\e(x,v,t), u^n_\e(x,t) \rt)|_{t=0} = \lt(f_{0,\e}(x,v), u_{0,\e}(x)\rt) \quad \mbox{for all} \quad n \geq 1, 
\]
and
\[
\lt(f^0_\e(x,v,t), u^0_\e(x,t)\rt) = (f_{0,\e}(x,v), u_{0,\e}(x)), \quad (x,v,t) \in \T^3 \times \R^3 \times (0,T).
\]
Before we consider (\ref{sys_rd}), we consider the existence of characteristics:
\begin{lemma}\label{cha}
For $u\in L^{\infty}(0,T; L^2(\T^3))$ such that $\|u\|_{L^\infty(0,T;L^2)}<\infty$ and a fixed $\e>0$, define the backward characteristic $Z_\e(s) := (X_\e(s), V_\e(s)) := (X_\e(s;t,x,v)$, $V_\e(s;t,x,v))$ by
\begin{align}\label{eqn_rtra0}
\begin{aligned}
\frac{d}{ds}X_{\e}(s)&=V_\e(s), \cr
\frac{d}{ds}V_\e(s)&=\eta_\e \star u(X_\e(s),s)-V_\e(s),
\end{aligned}
\end{align}
with the terminal datum
\[
X_\e(t)=x \quad \mbox{and} \quad V_\e(t)=v.
\]
Then $Z_\e(s)$ is globally well-defined and satisfies
\bq\label{est_z}
|Z_\e(s)| \leq C_{T,\e,u}(1 + |v|) \quad \mbox{and} \quad |\nabla_{x,v}Z_\e(s)|\leq C_{T,\e,u},
\eq
for some positive constant $C_{T,\e,u}=C\lt(T,\e, \|u\|_{L^\infty(0,T;L^2)}\rt)$.
\end{lemma}
\begin{proof} The existence part is clear  due to the regularization.
For the estimate of \eqref{est_z}, we rewrite $\eqref{eqn_rtra0}$ as
\begin{align}
\begin{split}\label{cha int}
X_\e(s)&=e^{t-s}x-\int^t_s e^{\tau-s}V_\e(\tau)\,d\tau,\cr
V_\e(s)&=e^{t-s}v-\int^t_s e^{\tau-s}(\eta_\e \star u)(X_\e(\tau),\tau)\,d\tau.
\end{split}
\end{align}
A straightforward computation yields
\[
|X_\e(s)| \leq |x| + \int_0^s |V_\e(\tau)|\,d\tau \leq C + \int_0^s |Z_\e(\tau)|\,d\tau
\]
and
$$\begin{aligned}
|V_\e(s)| &\leq e^{t-s}|v| + \int^t_s e^{\tau-s} \lt|(\eta_\e \star u)(X_\e(\tau),\tau)\rt|d\tau \leq C_T|v| + \frac{C_T}{\e^{3/2}}\|\eta\|_{L^2}\| u\|_{L^2},
\end{aligned}$$
where we used
\begin{equation*}
\|\eta_\e \star u\|_{L^\infty}\leq \|\eta_\e\|_{L^2}\| u\|_{L^2}\leq \frac{1}{\e^{3/2}}\|\eta\|_{L^2}\| u\|_{L^2},
\end{equation*}
 Thus we obtain
\[
|Z_\e(s)| \leq C_T|v| + C_{T,\e,u} + \int_0^s |Z_\e(\tau)|\,d\tau,
\]
which gives
\bq\label{est_z0}
|Z_\e(s)| \leq C_{T,\e,u}(1 + |v|),
\eq
for some positive constant  $C_{T,\e,u}$ depending on $T,\e$, and $\|u\|_{L^\infty(0,T;L^2)}$. Similarly, using
\[
\|\nabla_x (\eta_\e \star u)\|_{L^\infty} \leq \|\nabla_x \eta_\e\|_{L^2}\| u\|_{L^2}
\leq \frac{1}{\e^{5/2}}\|\nabla_x \eta\|_{L^2}\| u\|_{L^2},
\]
we get
$$\begin{aligned}
|\nabla_{x,v}X_\e(s)|&\leq C + \int^s_0|\nabla_{x,v}V_\e(\tau)|\,d\tau,\cr
|\nabla_{x,v}V_\e(s)|&\leq C_T+C_T\int^s_0|(\nabla_x\eta_\e) \star u(X_\e(\tau),\tau)||\nabla_{x,v}Z_\e(\tau)| \,d\tau\cr
&\leq C_T+ \frac{C_T}{\e^{5/2}}\int^s_0 \|\nabla_x \eta\|_{L^2}\| u\|_{L^2}|\nabla_{x,v}Z_\e(\tau)|\,d\tau.
\end{aligned}$$
Thus we have
\[
|\nabla_{x,v}Z_\e(s)|\leq C_T+ C_{T,\e}\int^s_0|\nabla_{x,v}Z_\e(\tau)|\,d\tau,
\]
which, from Gronwall's inequality, yields
\[
|\nabla_{x,v}Z_\e(s)| \leq C_{T,\e,u}.
\]
Here, $C_{T,\e,u}$ is a positive constant depending on $T,\e$, and $\|u\|_{L^\infty(0,T;L^2)}$.
\end{proof}
We now state the results on existence and uniqueness of the regularized and decoupled system \eqref{sys_rd}, and its uniform bound estimates in $n$ in the proposition below.
\begin{proposition}\label{prop_ext_rd}
Let $q>5$. For any $T>0$ and $n \in \N$, there exists a unique solution $(f^n_\e, u^n_\e)$ of the regularized and decoupled system \eqref{sys_rd} such that
$f^n_\e \in L^\infty(0,T;L^{\infty}_q(\T^3 \times \R^3))$ and
$u^n_\e\in \lt(H^1(0,T;L^2(\T^3)) \cap L^2(0,T;H^1(\T^3))\rt)$. Moreover, $(f^n_\e, u^n_\e)$ satisfies the following uniform-in-$n$ estimates:
\begin{itemize}
\item[(i)] $\|f^n_\e\|_{L^\infty(\T^3 \times \R^3 \times (0,T))} \leq C_1\|f_{0,\e}\|_{L^\infty(\T^3 \times \R^3)}$,
\item[(ii)] $\|u^n_\e\|_{L^\infty(0,T; L^2(\T^3)) \cap L^2(0,T;H^1(\T^3))}<C_{2,\e} ,\quad \|\pa_t u^n_\e\|_{L^2(\T^3 \times (0,T))}\leq C_{3,\e}$,
\item[(iii)] $\|f^n_{\e}\|_{L^\infty(0,T;L^{\infty}_q(\T^3 \times \R^3))}+\|\nabla_{x,v}f^n_{\e}\|_{L^\infty(0,T;L^{\infty}_q(\T^3 \times \R^3))} \leq C_{4,\e}$,
\item[(iv)] $\rho_{f^n_{\e}}+|U_{f^n_{\e}}|+ T_{f^n_{\e}}<C_{5,\e}$, \quad$\rho_{f^n_{\e}}, T_{f^n_{\e}}>C_{6,\e}$,
\end{itemize}
Here, $C_{1}=C_1(T)$ depends only on $T$, whereas $C_{2,\e}=C_2(T,f_0,u_0,\e)$, $C_{3,\e}=C_3(T,f_0,u_0,\nabla u_0,\e)$
  and $C_{i,\e}=C_i(T,f_0,\e)~(i=4,5,6)$.
\end{proposition}
\begin{remark}\label{rmk_f_inf} The upper bound estimate of $f^n_\e$ in $L^\infty(\T^3 \times \R^3 \times (0,T))$ does not depend on both $\e$ and $n$.
\end{remark}

%
%
%
\begin{proof}
We prove this proposition using induction.  The case $n=0$ is trivially satisfied. Assume that we have obtained  $(f^n, u^n) \in L^\infty(\T^3 \times \R^3 \times (0,T)) \times L^\infty(0,T; L^2(\T^3))$
that satisfies all the statement of Proposition \ref{sys_rd}.\newline
\noindent{\bf (1) Existence and uniqueness of $(f^{n+1}_\e, u^{n+1}_\e)$:}
 Under the assumption $(f^n, u^n) \in L^\infty(\T^3 \times \R^3 \times (0,T)) \times L^\infty(0,T; L^2(\T^3))$,
 $\eqref{sys_rd}_1$ can be seen as an inhomogeneous transport equation:
\begin{align}\label{linear trans}
\pa_t f^{n+1}_\e + v \cdot \nabla_x f^{n+1}_\e + (\eta_\e \star u^n_\e-v)\cdot\nabla_vf^{n+1}_\e-2f^{n+1}_\e =
\mathcal{M}(f^n_\e).
\end{align}
Thus, in view of the uniform bound on $\mathcal{M}(f^n_\e)$ given by Lemma \ref{M<f}, the  existence  follows straightforwardly once well-posedness of the  characteristic:
\[Z^{n+1}_\e(s) := (X^{n+1}_\e(s), V^{n+1}_\e(s)) := (X^{n+1}_\e(s;t,x,v), V^{n+1}_\e(s;t,x,v))\]
defined by
\begin{align}\label{eqn_rtra1}
\begin{aligned}
\frac{d}{ds}X^{n+1}_{\e}(s)&=V^{n+1}_\e(s), \quad 0 \leq s \leq T, \\[1mm]
\frac{d}{ds}V^{n+1}_\e(s)&=\eta_\e \star u^{n}_\e(X^{n+1}_\e(s),s)-V^{n+1}_\e(s),
\end{aligned}
\end{align}
with the terminal datum
\[
X^{n+1}_\e(t)=x \quad \mbox{and} \quad V^{n+1}_\e(t)=v,
\]
is verified, which is provided by Lemma \ref{cha} below.

On the other hand the assumption $(f^n, u^n) \in L^\infty(\T^3 \times \R^3 \times (0,T)) \times L^\infty(0,T; L^2(\T^3))$ together with the high-velocity cut-off function $\gamma_\e(v)$ implies that the drag forcing term in the fluid part belongs to $L^2(\T^3 \times (0,T))$ at least. Thus, by a standard existence theory of incompressible Navier-Stokes equations with a mollified convection term, we can obtain the global-in-time existence and uniqueness of solution $u_\e^{n+1}$ solving the fluid part in \eqref{sys_rd} with the regularity mentioned in Proposition \ref{prop_ext_rd}.\newline

\noindent{\bf (2) Uniform bound estimates in $n$: }\label{ssec_uni_n} We now prove the uniform-in-$n$ bounds  in Proposition \ref{prop_ext_rd}. 

\noindent$\bullet$ \emph{\bf Estimate of $\|f^n_\e(t)\|_{L^\infty}$}:  Integrating (\ref{linear trans}) along the characteristic defined in \eqref{eqn_rtra1}, we get  the mild form:
\begin{align}\label{mild e}
f^{n+1}_\e(x,v,t)=e^{2t}f_{0,\e}(Z^{n+1}_\e(0))+\int^t_0e^{2(t-s)}\mathcal{M}(f^n_\e)\lt(Z^{n+1}_\e(s),s \rt)ds.
\end{align}
Then, Lemma \ref{M<f} gives
$$\begin{aligned}
\|f^{n+1}_\e(t)\|_{L^\infty} &\leq \|f_{0,\e}\|_{L^\infty}e^{2T} + e^{2T}\int_0^t\|\mathcal{M}(f^n_\e)(s)\|_{L^\infty}\,ds\cr
 &\leq C_T \|f_{0,\e}\|_{L^\infty}  + C_T \int_0^t \|f^n_\e(s)\|_{L^\infty}\,ds.
\end{aligned}$$
Therefore,
\bq\label{est_fni}
\sup_{0 \leq t \leq T}\|f^n_\e(t)\|_{L^\infty} \leq C_T\|f_{0,\e}\|_{L^\infty} \quad \mbox{for} \quad n \geq 1.
\eq
\noindent$\bullet$ \emph{\bf Estimate of $\|u^n_\e(t)\|_{L^\infty(0,T:L^2)}$
and $\|\partial_tu^n_\e(t)\|_{L^2(0,T:L^2)}$}: Multiplying \eqref{sys_rd} by $u^{n+1}_\e$ and integrating it over $\T^3$ gives
\bq\label{est_un1}
\frac12\frac{d}{dt}\int_{\T^3} |u^{n+1}_\e|^2\,dx + \mu \int_{\T^3} |\nabla u^{n+1}_\e|^2\,dx = -\int_{\T^3 \times \R^3} \gamma_\e(v)(u^n_\e - v)f^n_\e \cdot u^{n+1}_\e\,dxdv
\eq
due to
\[
\int_{\T^3} (\eta_\e \star u^{n+1}_\e)\cdot\nabla_x u^{n+1}_\e \cdot u^{n+1}_\e\,dx = 0.
\]
On the other hand, the term on the right hand side of \eqref{est_un1} can be estimated as
$$\begin{aligned}
\lt|\int_{\T^3 \times \R^3} \gamma_\e(v)(u^n_\e - v)f^n_\e \cdot u^{n+1}_\e\,dxdv\rt|
&\leq C_\e \|f^n_\e\|_{L^\infty}\lt(1 + \|u^n_\e\|_{L^2}^2 + \|u^{n+1}_\e\|_{L^2}^2 \rt)\cr
& \leq C_{T,\e}\lt(1  + \|u^{n}_\e\|_{L^2}^2+\|u^{n+1}_\e\|_{L^2}^2 \rt),
\end{aligned}$$
thanks to \eqref{est_fni} and the cut-off function $\gamma_\e$. Thus we have
\[
\frac12\frac{d}{dt} \|u^{n+1}_\e\|_{L^2}^2 + \mu \|\nabla u^{n+1}_\e\|_{L^2}^2  \leq C_{T,\e}
\lt(1 + \|u^{n}_\e\|_{L^2}^2+\|u^{n+1}_\e\|_{L^2}^2 \rt),
\]
and this gives the uniform bound of $u^n_\e$ in $L^\infty(0,T;L^2(\T^3)) \cap L^2(0,T;H^1(\T^3))$.
Now we turn to the estimate of $\|\partial_tu^n_\e(t)\|_{L^2(0,T:L^2)}$. For this, we multiply \eqref{sys_rd} by $\partial_tu^n_\e(t)$, integrate over $x$,
  and use a similar argument as above to  derive
$$\begin{aligned}
\int_{\T^3} |\pa_t u^{n+1}_\e|^2\,dx + \frac\mu2 \frac{d}{dt}\int_{\T^3} |\nabla_x u^{n+1}_\e|^2\,dx &= -\int_{\T^3 \times \R^3} \gamma_\e(v)(u^n_\e - v)f^n_\e \cdot \pa_t u^{n+1}_\e\,dxdv\cr
&\leq C_\e\|\pa_t u^{n+1}_\e\|_{L^2}\cr
&\leq C_\e + \frac12\|\pa_t u^{n+1}_\e\|_{L^2}^2.
\end{aligned}$$
Integrating the above inequality with respect to time, we obtain
\[
\|\pa_t u^{n+1}_\e\|_{L^2(0,T;L^2)}^2 + \mu \|\nabla_x u^{n+1}_\e\|_{L^\infty(0,T;L^2)}^2 \leq C_\e T + \mu\|\nabla_x u_{0,\e}\|_{L^2}^2,
\]
which gives $\|\pa_t u^{n+1}_\e\|_{L^2(0,T;L^2)} \leq C(\e)$. \newline

\noindent$\bullet$ \emph{\bf Estimate of $\|f^n_{\e}\|_{L^\infty(0,T;L^{\infty}_q)}+\|\nabla_{x,v}f^n_{\e}\|_{L^\infty(0,T;L^{\infty}_q)}$}:  Let us take $C_{T,\e} > 0$ such that
\[
\frac{e^{2T}}{\e^{3/2}}\int^T_0\|\eta\|_{L^2}\| u^n_{\e}\|_{L^2}d\tau  \leq C_{T,\e}.
\]
Note that the constant above $C_{T,\e}$ does not depend on $n$ due to the uniform bound estimate of $u^n_\e$ in the previous part. Then it follows from \eqref{cha int} that
\[
|V^{n+1}_\e(t)| \geq |v| - \frac{e^{2T}}{\e^{3/2}}\int^T_0\|\eta\|_{L^2}\| u^n_{\e}\|_{L^2}d\tau \geq |v|  - C_{T,\e},
\]
that is,
\[
1 + C_{T,\e} + |V^{n+1}_\e(t)| \geq 1 + |v| \quad \mbox{for} \quad n \geq 1 \quad \mbox{and} \quad 0 \leq t \leq T.
\]
Using the above estimate, we find
$$\begin{aligned}
f_{0,\e}(Z^{n+1}_\e(0))&= f_{0,\e}(Z^{n+1}_\e(0))(1 + C_{T,\e} + |V^{n+1}_\e(0)|)^{q} (1 + C_{T,\e} + |V^{n+1}_\e(0)|)^{-q}\cr
&\leq C_{T,\e,q}\|f_{0,\e}\|_{L^{\infty}_q}(1+|v|)^{-q},
\end{aligned}$$
for $0 < q < \infty$. Similarly, with the aid of Lemma \ref{M<f}, we estimate
$$\begin{aligned}
\mathcal{M}(f^n_\e)(Z^{n+1}_\e(\tau),\tau)
&\leq\mathcal{M}(f^n_\e)(Z^{n+1}_\e(\tau),\tau)(1 + C_{T,\e} +|V^{n+1}_\e(\tau)|)^{q}
(1 + C_{T,\e}+|V^{n+1}_\e(\tau)|)^{-q}\cr
&\leq C_{T,\e,q}\|\mathcal{M}(f^n_\e)\|_{L^{\infty}_q}(1+|v|)^{-q}\cr
&\leq C_{T,\e,q}\|f^n_\e\|_{L^{\infty}_q}(1+|v|)^{-q}.
\end{aligned}$$
Combining all the above estimate, we have
\begin{align*}
|f^{n+1}_\e(x,v,t)|\leq C_{T,\e,q}\|f_{0,\e}\|_{L^{\infty}_q}(1+|v|)^{-q}+C_{T,\e,q}\int^t_0\|f^n_\e(s)\|_{L^{\infty}_q}(1+|v|)^{-q}ds.
\end{align*}
This readily gives
\bq\label{sum1}
\|f^{n+1}_\e(t)\|_{L^{\infty}_q} \leq C_{T,\e,q}\|f_{0,\e}\|_{L^{\infty}_q}+C_{T,\e,q}\int^t_0 \|f^n_\e(s)\|_{L^{\infty}_q} \,ds.
\eq
We next estimate the first-order derivative for $f^{n+1}_\e$. Note that the estimate in Lemma \ref{cha} is now uniform in $n$ due to the uniform bound estimate of $u^n_\e$ in $L^\infty(0,T;L^2(\T^3))$. This and using the similar argument as the above yield
$$\begin{aligned}
&|\nabla_{x,v}f^{n+1}_\e(x,v,t)|\cr
&\qquad\leq e^{2t}|\nabla_{x,v}f_{0,\e}(Z^{n+1}_\e(0))||\nabla_{x,v}Z^{n+1}_\e(0)| \cr
&\qquad +\int^t_0e^{2(t-s)}|\nabla_{x,v}\mathcal{M}(f^n_\e)\big(Z^{n+1}_\e(s),s\big)||\nabla_{x,v}Z^{n+1}_\e(s)|\,ds\cr
&\qquad \leq C_{T,\e}\|\nabla_{x,v}f_{0,\e}\|_{L^{\infty}_q}(1+|v|)^{-q}+C_{T,\e}\int^t_0\|\nabla_{x,v}\mathcal{M}(f^n_\e)\|_{L^{\infty}_q}(1+|v|)^{-q}\,ds\cr
&\qquad\leq C_{T,\e}\|\nabla_{x,v}f_{0,\e}\|_{L^{\infty}_q}(1+|v|)^{-q}+C_{T,\e}\int^t_0\left(\|f^n_\e\|_{L^{\infty}_q}+\|\nabla_{x,v}f^n_\e\|_{L^{\infty}_q}\right)
(1+|v|)^{-q}\,ds.
\end{aligned}$$
Hence we obtain
\begin{align}\label{sum2}
\|\nabla_{x,v}f^{n+1}_\e\|_{L^{\infty}_q}\leq
C_{T,\e}\|\nabla_{x,v}f_{0,\e}\|_{L^{\infty}_q}+C_{T,\e}\int^t_0\left(\|f^n_\e\|_{L^{\infty}_q}+\|\nabla_{x,v}f^n_\e\|_{L^{\infty}_q}\right)ds.
\end{align}
Combining (\ref{sum1}) and (\ref{sum2}), we have
$$\begin{aligned}
&\|f^{n+1}_\e(t)\|_{L^{\infty}_q}+\|\nabla_{x,v}f^{n+1}_\e(t)\|_{L^{\infty}_q}\cr&\qquad \leq C_{T,\e}\lt(\|f_{0,\e}\|_{L^{\infty}_q}+\|\nabla_{x,v}f_{0,\e}\|_{L^{\infty}_q}\rt)+C_{T,\e}\int^t_0\left(\|f^n_\e\|_{L^{\infty}_q}+\|\nabla_{x,v}f^n_\e\|_{L^{\infty}_q}\right)ds,
\end{aligned}$$
which yields the desired result.  \newline

\noindent$\bullet$ \emph{\bf Estimates of macroscopic fields of $f^n_\e$}: We show that macroscopic fields of $f$ satisfy $\rho_{f^n_{\e}}+|U_{f^n_{\e}}|+ T_{f^n_{\e}}<C_{T,\e}$ and $\rho_{f^n_{\e}}, T_{f^n_{\e}}>C_{T,\e}$ for some positive constant $C_{T,\e}$. For this, we take into account the integration of \eqref{mild e} and recall how we regularized $f_0$ to see
\[
\intr f^{n}_\e(x,v,t)\,dv \geq e^{2t}\intr f_0(Z^n_\e(0))\,dv \geq \intr \e e^{-|V^n_\e(0)|^2}dv \geq \intr \e e^{-C_{T,\e}(1+|v|)^2}dv \geq C_{T,\e},
\]
where we used \eqref{est_z0} together with the uniform estimate of $u^n_\e$. This gives the lower bound for $\rho_{f^n_\e}$. Then, the lower bound for $T_{f^n_\e}$ follows directly from Lemma \ref{moments}. The upper bounds can be easily obtained.
\end{proof}
%
%
%
%
\subsection{Proof of Proposition \ref{prop_ext_reg}. (1)} We are now ready to prove the existence
and uniqueness of $(f_\e, u_\e)$ stated in Proposition \ref{prop_ext_reg}. (1).
We split the proof into four steps as follows. \newline
\noindent\emph{\bf Step A.- Cauchy estimate for $f^n$}:  It follows from \eqref{mild e} that
$$\begin{aligned}
f^{n+1}_\e(x,v,t)-f^{n}_\e(x,v,t)&=\int^t_0e^{2(t-s)}\lt(\mathcal{M}(f^n_\e)\lt(Z^{n+1}_\e(s),s\rt) -\mathcal{M}(f^{n-1}_\e)\lt(Z^n_\e(s),s\rt)\rt)ds\cr
&=\int^t_0e^{2(t-s)}\lt(\mathcal{M}(f^n_\e)\lt(Z^{n+1}_\e(s),s\rt)
-\mathcal{M}(f^{n}_\e)\lt(Z^n_\e(s),s\rt)\rt)ds\cr
&\quad +\int^t_0e^{2(t-s)}\lt(\mathcal{M}(f^n_\e)\lt(Z^n_\e(s),s\rt)
-\mathcal{M}(f^{n-1}_\e)\lt(Z^n_\e(s),s\rt)\rt)ds\cr
&=: I_1 + I_2,
\end{aligned}$$
where $I_i,(i=1,2)$ can be estimate as follows.
$$\begin{aligned}
I_1 &=\int_0^t e^{2(t-s)} \nabla_{x,v} \mathcal{M}(f^n_\e)\lt( \alpha Z^{n+1}_\e(s) + (1-\alpha)Z^n_\e(s),s \rt) \cdot (Z^{n+1}_\e(s) - Z^n_\e(s))\,ds \cr
&\leq C_{T,\e}\int_0^t \|\nabla_{x,v}\mathcal{M}(f^n_\e)\|_{L^\infty_q}|Z^{n+1}_\e(s) - Z^n_\e(s)|\,ds (1 + |v|)^{-q} \cr
&\leq C_{T,\e} (1 + |v|)^{-q} \int_0^t \lt(\|f^n_\e(s)\|_{L^\infty_q} + \|\nabla_{x,v} f^n_\e(s)\|_{L^\infty_q} \rt) |Z^{n+1}_\e(s) - Z^n_\e(s)|\,ds, \cr
I_2 &\leq C_{T,\e}(1+|v|)^{-q} \int_0^t \|(f^n_\e-f^{n-1}_\e)(s)\|_{L^{\infty}_q}\,ds.
\end{aligned}$$
Here we used Lemma \ref{h} and the similar argument as in the proof of Proposition \ref{prop_ext_rd}. This yields
\[
\|(f^{n+1}_\e-f^{n}_\e)(t)\|_{L^{\infty}_q} \leq C_{T,\e}\int^t_0\|(f^n_\e-f^{n-1}_\e)(s)\|_{L^{\infty}_q}\,ds +C_{T,\e}\int^t_0|Z^{n+1}_\e(s)-Z^n_\e(s)|\,ds.
\]
\noindent\emph{\bf Step B.- Cauchy estimate for the characteristic $Z^{n+1}_\e$}: We first find from \eqref{eqn_rtra1} that
\[
|X^{n+1}_\e(s)-X^n_\e(s)| \leq \int^t_s |V^{n+1}_\e(\tau)-V^n_\e(\tau)|\,d\tau.
\]
We next estimate the characteristic for velocity as
$$\begin{aligned}
|V^{n+1}_\e(s)-V^n_\e(s)|&\leq\int^t_se^{\tau-s}
\left|\eta_\e \star u^n_\e(X^{n+1}_\e(\tau),\tau)-\eta_\e \star u^{n}_\e(X^n_\e(\tau),\tau)\right|d\tau  \cr
&\quad +\int^t_se^{\tau-s} \left|\eta_\e \star u^{n}_\e(X^n_\e(\tau),\tau)-\eta_\e \star u^{n-1}_\e(X^n_\e(\tau),\tau)\right|d\tau\cr
&\leq C_{T,\e}\int^t_s
\|\nabla\eta_\e\|_{L^2} \| u^n_\e\|_{L^2}|X^{n+1}_\e(\tau)-X^n_\e(\tau)|d\tau\cr
& \quad+C_T\int^T_0\|\eta_\e\|_{L^2} \| (u^n_\e-u^{n-1}_\e)(\tau)\|_{L^2}d\tau\cr
&\leq C_{T,\e}\int^T_0 |X^{n+1}_\e(\tau)-X^n_\e(\tau)| +\| (u^n_\e-u^{n-1}_\e)(\tau)\|_{L^2}\,d\tau,
\end{aligned}$$
where we used the uniform bound estimate of $\|u^n_\e\|_{L^\infty(0,T;L^2)}$ in $n$. Thus we have
$$\begin{aligned}
|Z^{n+1}_\e(s)-Z^n_\e(s)|
\leq C_{T,\e}\int^T_0 |Z^{n+1}_\e(\tau)-Z^n_\e(\tau)| \,d\tau+C_{T,\e}\int^T_0\| (u^n_\e-u^{n-1}_\e)(\tau)\|_{L^2}d\tau.
\end{aligned}$$
\noindent\emph{\bf Step C.- Cauchy estimate for the fluid velocity $u^n$}: For notational simplicity, we set $w^{n+1}_\e := u^{n+1}_\e - u^n_\e$. Then it follows from $\eqref{sys_rd}_2$ that $w^{n+1}$ satisfies
\begin{align}\label{est_un}
\begin{aligned}
&\pa_t w^{n+1}_\e + (\eta_\e \star w^{n+1}_\e) \cdot \nabla_x u^{n+1}_\e + (\eta_\e \star u^n_\e) \cdot \nabla_x w^{n+1}_\e + \nabla_x (p^{n+1}_\e - p^n_\e) - \mu \Delta_x w^{n+1}_\e \cr
&\qquad = -\intr \gamma_\e(v)w^n_\e f^n_\e \,dv -\intr \gamma_\e(v) (u^{n-1}_\e- v) (f^n_\e - f^{n-1}_\e)\,dv
\end{aligned}
\end{align}
and $\nabla_x \cdot w^{n+1}_\e = 0$. Multiplying \eqref{est_un} by $w^{n+1}_\e$ and integrating it over $\T^3$ gives
$$\begin{aligned}
&\frac12\frac{d}{dt}\|w^{n+1}_\e\|_{L^2}^2 + \mu \|\nabla_x w^{n+1}_\e\|_{L^2}^2 \cr
&\qquad = -\int_{\mathbb{T}^3} (\eta_\e \star w^{n+1}_\e) \cdot \nabla_x u^{n+1}_\e \cdot w^{n+1}_\e \,dx - \int_{\T^3 \times \R^3} \gamma_\e(v)w^n_\e \cdot w^{n+1}_\e f^n_\e\,dxdv\cr
&\qquad \quad - \int_{\T^3 \times \R^3} \gamma_\e(v) (u^{n-1}_\e -v) \cdot w^{n+1}_\e(f^n_\e - f^{n-1}_\e)\,dxdv\cr
&\qquad =: J_1 + J_2 + J_3,
\end{aligned}$$
thanks to
\[
\int_{\T^3} (\eta_\e \star u^n_\e) \cdot \nabla_x w^{n+1}_\e \cdot w^{n+1}_\e\,dx = 0.
\]
We then estimate $J_i~ (i=1,2,3)$ as
$$\begin{aligned}
J_1 &= \int_{\mathbb{T}^3} (\eta_\e \star w^{n+1}_\e) \cdot \nabla_x w^{n+1}_\e \cdot u^{n+1}_\e \,dx \leq C_\e\|u^{n+1}_\e\|_{L^2}\|w^{n+1}_\e\|_{L^2}\|\nabla_x w^{n+1}_\e\|_{L^2},\cr
J_2 &\leq C_\e\|f^n_\e\|_{L^\infty}\|w^n_\e\|_{L^2} \|w^{n+1}_\e\|_{L^2},\cr
J_3 &\leq C_\e (1 + \|u^{n-1}_\e\|_{L^2})\|w^{n+1}_\e\|_{L^2}\|f^n_\e - f^{n-1}_\e\|_{L^\infty}.
\end{aligned}$$
This, together with the uniform bound estimate of $(f^n_\e, u^n_\e)$ in $n$ implies
\[
\frac{d}{dt}\|w^{n+1}_\e\|_{L^2}^2 + \mu \|\nabla_x w^{n+1}_\e\|_{L^2}^2  \leq C_{T,\e}\lt(\|w^n_\e\|_{L^2}^2 + \|w^{n+1}_\e\|_{L^2}^2 + \|f^n_\e - f^{n-1}_\e\|_{L^\infty}^2\rt).
\]
\noindent\emph{\bf Step D.- Cauchy estimate for $(f^n_\e, u^n_\e, Z^n_\e)_{n \in \N}$}: Combining the estimates in previous steps, we have
$$\begin{aligned}
&\|f^{n+1}_\e(t)-f^{n}_\e(t)\|_{L^{\infty}_q}+\|Z^{n+1}_\e(t)-Z^n_\e(t)\|_{L^\infty}+\|u^{n+1}_\e(t)-u^n_\e(t)\|_{L^2}\cr
&\hspace{1cm}\leq C_{T,\e} \int^T_0\|f^{n}_\e(\tau)-f^{n-1}_\e(\tau)\|_{L^{\infty}_q}+\|Z^{n+1}_\e(\tau)-Z^n_\e(\tau)\|_{L^\infty}+\|u^n_\e(\tau)-u^{n-1}_\e(\tau)\|_{L^2}\,d\tau,
\end{aligned}$$
from which we can conclude that$(f^n_\e, u^n_\e)_{n \in \N}$ is a Cauchy sequence in $L^\infty(0,T;L^\infty_q(\T^3 \times \R^3))) \times L^\infty(0,T;L^2(\T^3))$. Therefore, for a fixed $\e > 0$, there exist limiting functions $f_\e,u_\e, Z_\e$ such that
\begin{align*}
\sup_{0 \leq t \leq T}\lt(\|f^{n}_\e(t)-f_\e(t)\|_{L^{\infty}_q}+\|Z^n_\e(t)-Z_\e(t)\|_{L^\infty}+\|u^n_\e(t)-u_\e(t)\|_{L^2}\rt)\rightarrow 0 \quad \mbox{as} \quad n \to \infty.
\end{align*}
Then, by a standard argument, we can easily show that $(f_\e,Z_\e,u_\e)$ solve the regularized system \eqref{sys_reg}.
%
%

%
%
%
%
%
\section{Proof of Proposition \ref{prop_ext_reg}. (2): Uniform-in-$\e$ estimates on $(f_\e, u_\e)$ }\label{sec_pro2}
In this section, we establish several uniform-in-$\e$ estimates for $(f_\e, u_\e)$ given in Proposition \ref{prop_ext_reg}. (2). For notational simplicity, we drop the subscript $f$ in $\rho_{f_\e}, U_{f_\e}$, and $T_{f_\e}$ when there is no confusion, i.e., we denote by $\rho_\e := \rho_{f_\e}$, $U_\e:= U_{f_\e}$, and $T_\e := T_{f_\e}$.

\noindent $\bullet$ {\bf Uniform bounds of the total energy:} A straightforward computation yields from $\eqref{sys_reg}_1$ that
\[
\frac{d}{dt}M_2 f_\e + 2M_2 f_\e \leq 2\int_{\T^3} |\eta_\e \star u_\e(x,t)| m_1 f_\e\,dx.
\]
This together with Lemma \ref{lem_mom}; $m_1f_\e\leq C(m_2f_\e)^{4/5}$, Minkowski's inequality; $\|\eta_\e \star u_\e (t)\|_{L^5}\leq C\|u_\e (t)\|_{L^5}$, and the uniform bound estimate of $\|f_\e\|_{L^\infty}$ in Proposition \ref{prop_ext_rd}, (see also Remark \ref{rmk_f_inf}) gives
\begin{equation}\label{M2}
\frac{d}{dt}M_2 f_\e + 2M_2 f_\e \leq \|\eta_\e \star u_\e (t)\|_{L^5}\|m_1 f_\e(t)\|_{L^{5/4}} \leq C\|u_\e(t)\|_{L^5}(M_2 f_\e)^{4/5}.
\end{equation}
Applying Gronwall's inequality, we obtain
\bq\label{est_fe}
M_2f_\e(t) \leq  C\lt( (M_2 f_{0,\e})^{1/5} + \int_0^t \|u_\e(s)\|_{L^5}\,ds \rt)^5 \leq C\lt(1  + \int_0^t \|u_\e(s)\|_{L^5}\,ds \rt)^5,
\eq
due to $M_2 f_{0,\e} \leq CM_2 f_0$, where $C > 0$ is independent of $\e$. We next turn to the uniform estimate of the fluid velocity. For this, we multiply $\eqref{sys_reg}_2$ by $u_\e$, integrate over $x$ to get
\begin{align*}
\begin{split}
\frac12\frac{d}{dt}\|u_\e\|_{L^2}^2 + \mu\|\nabla_x u_\e\|_{L^2}^2 &= -\int_{\T^3 \times \R^3} f_\e u_\e \cdot (u_\e - v)\gamma_\e(v)\,dxdv\cr
&= -\int_{\T^3 \times \R^3} f_\e |u_\e|^2 \gamma_\e(v)\,dxdv + \int_{\T^3 \times \R^3} f_\e u_\e \cdot v \gamma_\e(v)\,dxdv\cr
&\leq \int_{\T^3} |u_\e|m_1 f_\e\,dx.
\end{split}
\end{align*}
Then, by using the argument in (\ref{M2}) and (\ref{est_fe}), we can bound the last term as
\begin{align}\label{here}
\begin{split}
\int_{\T^3} |u_\e|m_1 f_\e\,dx&\leq \|u_\e\|_{L^5}\|m_1 f_\e\|_{L^{5/4}}\cr
&\leq  C\|u_\e(t)\|_{L^5}(M_2 f_\e)^{4/5}\cr
&\leq C\|u_\e\|_{L^5}\lt(1 + \int_0^t \|u_\e(s)\|_{L^5}\,ds \rt)^4\cr
&\leq C\|u_\e\|_{H^1}\lt(1 + \lt(\int_0^t \|u_\e(s)\|_{H^1}^2\,ds\rt)^{1/2} \rt)^4,
\end{split}
\end{align}
where we used the Sobolev embedding $L^5(\T^3) \hookrightarrow H^1(\T^3)$ in the last line.
We then use the Young's inequality to get
$$\begin{aligned}
&\frac12\frac{d}{dt}\|u_\e\|_{L^2}^2 + \mu\|\nabla_x u_\e\|_{L^2}^2\cr
 &\hspace{1.5cm}\leq C + C\|u_\e\|_{L^2}^2 + \frac{\mu}{2}\|\nabla_x u_\e\|_{L^2}^2 +  C\int_0^t  \|u_\e(s)\|_{L^2}^8 \,ds+ C\lt(\int_0^t \|\nabla_x u_\e(s)\|_{L^2}^2 \,ds\rt)^4.
\end{aligned}$$
Integrating the above inequality over the time interval $[0,t]$, we find
$$\begin{aligned}
\|u_\e\|_{L^2}^2 + \mu\int_0^t \|\nabla_x u_\e(s)\|_{L^2}^2\,ds &\leq \|u_{0,\e}\|_{L^2}^2 + C + C\int_0^t \|u_\e(s)\|_{L^2}^2\,ds + C\int_0^t \|u_\e(s)\|_{L^2}^8\,ds \cr
&\quad + C\int_0^t \lt(\int_0^s \|\nabla_x u_\e(\tau)\|_{L^2}^2 \,d\tau \rt)^4 ds.
\end{aligned}$$
We then apply the Gronwall's inequality to obtain that there exists a $0 < T_* \leq T$ such that
\bq\label{uni_ue}
\|u_\e(t)\|_{L^2}^2 + \int_0^t \|\nabla_x u_\e(s)\|_{L^2}^2\,ds \leq C \quad \mbox{for} \quad 0 \leq t \leq T_*,
\eq
due to $\|u_{0,\e}\|_{L^2} \leq C\|u_0\|_{L^2}$, where $C>0$ is independent of $\e$. We also combine \eqref{est_fe} and \eqref{uni_ue} to have
\begin{equation}\label{M2f bound}
M_2 f_\e(t) \leq C \quad \mbox{for} \quad 0 \leq t \leq T_*,
\end{equation}
where $C>0$ is independent of $\e$.\newline
\noindent $\bullet$ {\bf Uniform bound of $\|(\eta_\e \star u_\e - v)(1+|v|)f_\e\|_{L^1}$:} We divide the integral as
\begin{align*}
&\|(\eta_\e \star u_\e - v)(1+|v|)f_\e\|_{L^1}\cr
&\qquad=\int_0^{T_*} \int_{\T^3 \times \R^3} |(\eta_\e \star u_\e - v)|f_\e\,dxdvdt +\int_0^{T_*} \int_{\T^3 \times \R^3} |(\eta_\e \star u_\e - v)||v|f_\e\,dxdvdt\cr
&\qquad=: I_1+I_2,
\end{align*}
and estimate $I_1$ and $I_2$ separately. For the estimate of $I_1$, we first note that
$$\begin{aligned}
I_1 &\leq \int_0^{T_*} \int_{\T^3} |\eta_\e \star u_\e| \rho_\e\,dxdt + \int_0^{T_*} \int_{\T^3} m_1f_\e\,dxdt\cr
& \leq \int_0^{T_*} \|u_\e(t)\|_{L^{5/2}}\|\rho_\e(t)\|_{L^{5/3}}\,dt + C\int_0^{T_*} \|m_1 f_\e\|_{L^{5/4}}\,dt,
\end{aligned}$$
where the second term on the right hand side of the above inequality can be uniformly bounded as
$$\begin{aligned}
\|m_1 f_\e\|_{L^{5/4}}&\leq \lt(1 + \int_0^t \|u_\e(s)\|_{L^5}\,ds \rt)^4\cr
&\leq C\lt(1 + \lt(\int_0^t \|u_\e(s)\|_{H^1}^2\,ds\rt)^{1/2} \rt)^4\cr
& \leq C \quad \mbox{for} \quad t \in (0,T_*),
\end{aligned}$$
by using the same argument as in the estimate of the total energy. For the first term, we use Lemma \ref{lem_mom}; $m_0f_\e\leq C(m_1f_\e)^{3/4}$ to get $\|\rho_\e\|_{L^{5/3}} \leq C\|m_1 f_\e\|_{L^{5/4}}^{3/4}$, where $C>0$ is independent of $\e$. A similar argument as in the previous estimate then yields
$$\begin{aligned}
&\int_0^{T_*} \|u_\e(t)\|_{L^{5/2}}\|\rho_\e(t)\|_{L^{5/3}}\,dt\cr
&\hspace{1.5cm}\leq C\int_0^{T_*} \|u_\e(t)\|_{L^{5/2}}\|m_1 f_\e(t)\|_{L^{5/4}}^{3/4}\,dt\cr
&\hspace{1.5cm}\leq C\int_0^{T_*} \|u_\e(t)\|_{H^1}\lt(1 + \int_0^t \|u_\e(s)\|_{H^1}\,ds \rt)^3\,dt\cr
&\hspace{1.5cm}\leq C\int_0^{T_*} \|u_\e(t)\|_{H^1}\,dt + C\int_0^{T_*}  \|u_\e(t)\|_{H^1}\lt(\int_0^t \|u_\e(s)\|_{H^1}^2\,ds\rt)^{3/2} dt\cr
&\hspace{1.5cm}\leq C\int_0^{T_*} \|u_\e(t)\|_{H^1}\,dt\cr
&\hspace{1.5cm}\leq C\lt(\int_0^{T_*}\|u_\e(t)\|_{H^1}^2\,dt\rt)^{1/2}\cr
&\hspace{1.5cm}\leq C,
\end{aligned}$$
for $0 \leq t \leq T_*$, where $C >0$ is independent of $\e$ due to \eqref{uni_ue}. For $I_2$,
we decompose similarly as
\[
I_2 \leq \int^{T_*}_0\int_{\T^3}|\eta_\e\star u_\e| m_1f_\e \,dxdt + \int_0^{T_*} M_2 f_\e \,dt.
\]
The uniform boundedness of the second term on the right hand side is obtained in (\ref{M2f bound}). The computation for the first term is treated in \eqref{here}. This concludes the desired result.\newline
\noindent $\bullet$ {\bf Uniform bound of third moment:} We multiply (\ref{sys_reg}) by
\[
\Phi(x,v)=\frac{(1+|v|^2)^{1/2}x\cdot v}{(1+|x|^2)^{1/2}}
\]
and integrate on $\mathbb{T}^3\times\mathbb{R}^3 \times [0,T_*]$ to get
\begin{align*}
\begin{split}
-\int^{T_*}_0\int_{\T^3 \times \R^3}v\cdot \nabla_xf_\e\Phi \,dxdvdt
&=\int^{T_*}_0\int_{\T^3 \times \R^3}\partial_tf_\e\Phi \,dxdvdt\cr
&\quad +\int^{T_*}_0\int_{\T^3 \times \R^3}\nabla_v\cdot\left\{\big(\eta_\e\star u_\e-v\big)f_\e\right\}\Phi \,dxdvdt\cr
&\quad -\int^{T_*}_0\int_{\T^3 \times \R^3}\{\mathcal{M}(f_\e)-f_\e\}\Phi \,dxdvdt.
\end{split}
\end{align*}
We denote the left hand side by $L$ and the three terms on the right hand side by $R_i~ (i=1,2,3)$.\newline
\noindent $\diamond$ {\bf The estimate of $L$: } By divergence theorem, we have
\begin{align*}
\begin{split}
L&=\int^{T_*}_0\int_{\T^3 \times \R^3}f_\e v\cdot\nabla_x\Phi \,dxdvdt\cr
&=\int^{T_*}_0\int_{\T^3 \times \R^3}f_\e\{v(1+|v|^2)^{1/2}\}
\cdot\nabla_x\lt\{\frac{x\cdot v}{(1+|x|^2)^{1/2}}\rt\}dxdvdt\cr
&=\int^{T_*}_0\int_{\T^3 \times \R^3}f_\e\{v(1+|v|^2)^{1/2}\}
\cdot\lt\{\frac{v}{(1+|x|^2)^{1/2}}+\frac{-x(x\cdot v)}{(1+|x|^2)^{3/2}}\rt\}dxdvdt\cr
&=\int^{T_*}_0\int_{\T^3 \times \R^3}f_\e(1+|v|^2)^{1/2}
\cdot\lt\{\frac{|v|^2}{(1+|x|^2)^{1/2}}+\frac{-(x\cdot v)^2}{(1+|x|^2)^{3/2}}\rt\}dxdvdt\cr
&=\int^{T_*}_0\int_{\T^3 \times \R^3}f_\e\frac{|v|^2(1+|v|^2)^{1/2}}{(1+|x|^2)^{1/2}}
\lt\{1-\frac{(x\cdot v)^2}{(1+|x|^2)|v|^2}\rt\}dxdvdt.
\end{split}
\end{align*}
On the other hand, we observe
\[
1-\frac{(x\cdot v)^2}{(1+|x|^2)|v|^2}\geq 1-\frac{|x|^2|v|^2}{(1+|x|^2)|v|^2} =\frac{1}{1+|x|^2} \geq 1/4, \]
and
\[
\frac{|v|^2(1+|v|^2)^{1/2}}{(1+|x|^2)^{1/2}}\geq \frac{1}{2}|v|^3,
\]
for $(x,v) \in \T^3 \times \R^3$. This yields
\begin{align*}
L\geq \frac{1}{8}\int^{T_*}_0\int_{\mathbb{T}^3\times\mathbb{R}^3}|v|^3f_\e\,dxdvdt.
\end{align*}
\noindent$\diamond$ {\bf The estimate of $R_1$:} Since $\Phi$ does not depend on $t$, we can integrate in time as
\begin{align*}
R_1&=\int_{\mathbb{T}^3\times\mathbb{R}^3}\lt(f_\e({T_*})-f_\e(0)\rt)\Phi \,dxdv\cr&\leq \int_{\mathbb{T}^3\times\mathbb{R}^3}
\lt(f_\e({T_*})-f_\e(0)\rt)(1+|v|^2)\,dxdv\cr
& \leq C,
\end{align*}
where we used $\Phi(x,v)\leq (1+|v|^2)$ for $(x,v) \in \T^3 \times \R^3$ and \eqref{uni_te}, and the constant $C>0$ is independent of $\e$. \newline
\noindent$\diamond$ {\bf The estimate of $R_2$:} Using divergence theorem, we estimate
\begin{align*}
R_2&=\int^{T_*}_0\int_{\T^3 \times \R^3}\nabla_v\cdot\left\{\big(\eta_\e\star u_\e-v\big)f_\e\right\}\frac{(1+|v|^2)^{1/2}x\cdot v}{(1+|x|^2)^{1/2}}\,dxdvdt\cr
&=-\int^{T_*}_0\int_{\mathbb{T}^3\times\mathbb{R}^3}\left\{\big(\eta_\e\star u_\e-v\big)f_\e\right\}\cdot\nabla_v\lt\{\frac{(1+|v|^2)^{1/2}x\cdot v}{(1+|x|^2)^{1/2}}\rt\}dxdvdt\cr
&=-\int^{T_*}_0\int_{\mathbb{T}^3\times\mathbb{R}^3}\left\{\big(\eta_\e\star u_\e-v\big)f_\e\right\}
(1+|v|^2)^{1/2}\lt\{\frac{v\{x\cdot v\}+x(1+|v|^2)}{(1+|v|^2)(1+|x|^2)^{1/2}}\rt\}dxdvdt.
\end{align*}
Note that
\begin{align*}
\lt|\frac{v\{x\cdot v\}+x(1+|v|^2)}{(1+|v|^2)(1+|x|^2)^{1/2}}\rt|\leq 2 \quad \mbox{for} \quad (x,v) \in \T^3 \times \R^3,
\end{align*}
which gives
\begin{align*}
|R_2|\leq 2\int^{T_*}_0\int_{\T^3 \times \R^3}|\eta_\e\star u_\e-v|f_\e(1+|v|)\,dxdvdt = 2\|(\eta_\e \star u_\e - v)(1+|v|)f_\e\|_{L^1}  \leq C,
\end{align*}
where we used \eqref{est_fk}. \newline
\noindent$\diamond$ {\bf The estimate of $R_3$:} A straightforward computation gives
$$\begin{aligned}
R_3& \leq\int^{T_*}_0\int_{\T^3 \times \R^3}\big\{\mathcal{M}(f_\e)+f_\e\big\}(1+|v|^2)\,dxdvdt\cr
&= 2\int^{T_*}_0\int_{\T^3 \times \R^3} f_\e(1 + |v|^2) \,dxdvdt \cr
&\leq C_{f_0,u_0,T_*}.
\end{aligned}$$
Combining all these estimates, we obtain
\[
\int_0^{T_*}\int_{\mathbb{T}^3 \times\mathbb{R}^3}f_\e|v|^3\,dxdvdt \leq C_{f_0,u_0,T_*}.
\]
\noindent $\bullet$ {\bf Uniform bound of entropy:} Multiply  $\eqref{sys_reg}_1$ by $\ln f_\e$ and integrate with respect to $x$ and $v$ to get
\begin{align*}
&\frac{d}{dt} \int_{\mathbb{T}^3\times\mathbb{R}^3} f_\e\ln f_\e \,dxdv + \int_{\mathbb{T}^3\times\mathbb{R}^3} \lt(v\cdot\nabla_xf_\e\rt)\ln f_\e \,dxdv+\int_{\mathbb{T}^3\times\mathbb{R}^3} \nabla_v \cdot ((\eta_\e \star u_\e-v)f_\e)\ln f_\e \,dxdv\cr
 &\hspace{1cm}= \int_{\mathbb{T}^3\times\mathbb{R}^3} \left
(\mathcal{M}(f_\e) - f_\e\right)\ln f_\e \,dxdv.
\end{align*}
The second term on the left hand side vanishes due to the divergence theorem.  Using divergence theorem and integration by parts, we can estimate the third term on the left hand side as
\begin{align*}
\int_{\mathbb{T}^3\times\mathbb{R}^3} \nabla_v \cdot ((\eta_\e \star u_\e-v)f_\e)\ln f_\e \,dxdv &= -\int_{\mathbb{T}^3\times\mathbb{R}^3} (\eta_\e \star u_\e-v)\nabla_v f_\e \,dxdv \cr &=-3\int_{\mathbb{T}^3\times\mathbb{R}^3} f_\e \,dxdv.
\end{align*}
Since the local Maxwellian shares the same moments up to second order with $f_\e$, we get
\begin{align*}
\int_{\mathbb{T}^3\times\mathbb{R}^3} \left\{\mathcal{M}(f_\e) - f_\e\right\}\ln \mathcal{M}(f_\e)\,dxdv
&=\int_{\mathbb{T}^3\times\mathbb{R}^3} \left\{\mathcal{M}(f_\e) - f_\e\right\}\left\{\ln \frac{\rho_\e}{\sqrt{(2\pi T_\e)^3}}-\frac{|v-U_\e|^2}{2T_\e}\right\}dxdv\cr
&=0,
\end{align*}
which immediately gives
\[
\int_{\mathbb{T}^3\times\mathbb{R}^3} \left\{\mathcal{M}(f_\e) - f_\e\right\}\ln f_\e \,dxdv=-\int_{\mathbb{T}^3\times\mathbb{R}^3} \left\{\mathcal{M}(f_\e) - f_\e\right\}\left\{\ln \mathcal{M}(f_\e)-\ln f_\e\right\}dxdv \leq 0.
\]
Thus, we obtain
\begin{align*}
\frac{d}{dt} \int_{\mathbb{T}^3\times\mathbb{R}^3} f_\e\ln f_\e \,dxdv -3\int_{\mathbb{T}^3\times\mathbb{R}^3} f_\e \,dxdv=-\int_{\mathbb{T}^3\times\mathbb{R}^3} \left\{\mathcal{M}(f_\e) - f_\e\right\}\lt(\ln \mathcal{M}(f_\e)-\ln f_\e\rt) dxdv.
\end{align*}
Integrating in time, we get
$$\begin{aligned}
&\int_{\mathbb{T}^3\times\mathbb{R}^3} f_\e(t)\ln f_\e(t) \,dxdv+\int^t_0\int_{\mathbb{T}^3\times\mathbb{R}^3} \left\{\mathcal{M}(f_\e) - f_\e\right\}\lt(\ln \mathcal{M}(f_\e)-\ln f_\e\rt) dxdvds \cr
&\quad \leq \int_{\mathbb{T}^3\times\mathbb{R}^3} f_{0,\e}\ln f_{0,\e}\,dxdv + 3M_0f_0 T \quad \mbox{for} \quad t \in (0,T).
\end{aligned}$$
Then, it is standard to show that (see for example, \cite{CIP,DL})
\[
\sup_{0 \leq t \leq T}\int_{\mathbb{T}^3\times\mathbb{R}^3} f_\e(t)|\ln f_\e(t)| \,dxdv\leq C(f_0,T).
\]
This completes the proof.
%
%
%
%
%
%
\section{Global existence of weak solutions}\label{sec_gl_weak}
\subsection{Weak compactness  of $f_\e$ and $\mathcal{M}(f_\e)$} In this part, we use the uniform estimates in $\e$ obtained in the previous subsection to derive compactness of $(f_\e,u_\e)$ and the relaxation operators.

We have derived in the previous section that there exists a constant $C$, independent of $\e$ such that
\begin{eqnarray*}
\int^{T_*}_0\int_{\T^3 \times \R^3}(1+|v|^3+|\ln f_{\e}|)f_{\e} \,dxdvdt \leq C.
\end{eqnarray*}
Dunford-Pettis theorem  then implies that $f_{\e}$, $f_{\e}v$ and $f_{\e}|v|^2$ are  weakly compact in $L^1(\T^3\times\R^3 \times (0,{T_*}))$. To derive the weak compactness of $\mathcal{M}(f_\e)$, we
 compute for  $R>1$
\begin{align*}
\mathcal{M}(f)-f&=\big\{\mathcal{M}(f)-f\big\}1_{\mathcal{M}(f)<Rf}+\big\{\mathcal{M}(f)-f\big\}1_{\mathcal{M}(f)\geq Rf}\cr
&\leq(R-1)f1_{\mathcal{M}(f)<Rf} +\frac{1}{\ln R}\big(\mathcal{M}(f)-f\big)\big(\ln\mathcal{M}(f)-\ln f\big)1_{\mathcal{M}(f)\geq Rf},
\end{align*}
so that
\begin{align*}
\mathcal{M}(f)\leq Rf
+\frac{1}{\ln R}\big(\mathcal{M}(f)-f\big)\big(\ln\mathcal{M}(f)-\ln f\big).
\end{align*}
Now, we take an arbitrary measurable set $B_{x,v} \subseteq\mathbb{T}^3 \times \R^3$ and integrate over $B_{x,v}\times [0,{T_*}]$ to get
$$\begin{aligned}
&\int^{T_*}_0\int_{B_{x,v}}\mathcal{M}(f)\,dxdvdt\cr
&\hspace{1.2cm}\leq R\int^{T_*}_0\int_{B_{x,v}}f\,dxdvdt+\frac{1}{\ln R}\int^{T_*}_0\int_{ B_{x,v}}\big(\mathcal{M}(f)-f\big)\big(\ln\mathcal{M}(f)-\ln f\big)\,dvdxdt\cr
&\hspace{1.2cm}\leq R\int^{T_*}_0\int_{B_{x,v}}f\,dxdvdt+\frac{1}{\ln R}\int^{T_*}_0\int_{ \mathbb{T}^3 \times \mathbb{R}^3}\big(\mathcal{M}(f)-f\big)\big(\ln\mathcal{M}(f)-\ln f\big)\,dvdxdt\cr
&\hspace{1.2cm}\leq R\int^{T_*}_0\int_{B_{x,v}} f\,dxdvdt+\frac{1}{\ln R}\lt(\int_{\mathbb{T}^3 \times\mathbb{R}^3}f_0\ln f_0\,dxdv + 3 M_0 f_0 T\rt)\cr
&\hspace{1.2cm}\leq R\int^{T_*}_0\int_{B_{x,v}} f\,dxdvdt+\frac{1}{\ln R}\left(\int_{\mathbb{T}^3 \times\mathbb{R}^3}f_0|\ln f_0|dxdv+C_{f_0,T}\right).
\end{aligned}$$
Then, Dunford-Pettis theorem again gives the weak compactness of $\mathcal{M}(f)$ in $L^1( \mathbb{T}^3\times \mathbb{R}^3 \times (0,{T_*}))$.
%
%
%
%
\subsection{Strong compactness of $\rho_\e$, $U_\e$ and $T_\e$}
From the argument in the previous section, we see that there exists $f\in L^1( \mathbb{T}^3\times \mathbb{R}^3 \times (0,{T_*}))$ such that $f_{\e}$, $f_{\e}v$, $f_{\e}|v|^2$   converge to $f$, $fv$, $f|v|^2$ weakly in
$L^1( \mathbb{T}^3\times \mathbb{R}^3 \times (0,{T_*}))$
respectively, which also implies
\[
\rho_\e = \int_{\mathbb{R}^3} f_\e \,dv \rightharpoonup \int_{\mathbb{R}^3} f \,dv = \rho, \quad \rho_\e U_\e = \int_{\mathbb{R}^3} vf_\e \,dv \rightharpoonup \int_{\mathbb{R}^3} vf \,dv = \rho U,
\]
and
\[
3\rho_{\e}T_{\e}+\rho_{\e}|U_\e|^2 =\int_{\mathbb{R}^3}f_{\e}|v|^2dv \rightharpoonup \int_{\mathbb{R}^3}f|v|^2dv =3\rho T+\rho|U|^2
\]
 in $L^1(\mathbb{T}^3 \times (0,{T_*}))$.
Thanks to the velocity averaging lemma \cite{G-L-P-S},  the above convergence actually is strong, which gives the almost everywhere
convergence of the macroscopic fields:
\bq\label{ae}
\rho_{\e} \rightarrow \rho \quad \mbox{a.e on } \mathbb{T}^3\times [0,{T_*}], \quad
U_{\e } \rightarrow U \quad \mbox{a.e on } E, \quad \mbox{and} \quad
T_{\e}\rightarrow T \quad \mbox{a.e on }E,
\eq
where $E=\{(x,t)\in\mathbb{T}^3\times (0,{T_*})\, | \,\rho(x,t)\neq0\}$.
Next, we need to show that $\mathcal{M}(f_{\e})$ converges weakly in $L^1$ to $\mathcal{M}(f)$.
\subsection{$\mathcal{M}(f_\e)$ converges to $\mathcal{M} (f)$ in $L^1(\T^3\times\R^3 \times (0,{T_*}))$}
Since (\ref{ae}) implies
\begin{align*}
\mathcal{M}(\rho_{\e},U_{\e},T_{\e})\varphi\rightarrow  \mathcal{M}(\rho,U,T)\varphi \mbox{ a.e on } E\times\mathbb{R}^3
\end{align*}
for any non-negative $L^{\infty}$ function $\varphi$, we have from Fatou's lemma  that
\begin{align*}
\int_{E\times \mathbb{R}^3}\mathcal{M}(\rho,U,T)\varphi \,dxdvdt&\leq
\lim_{\e\rightarrow0}\int_{E\times \mathbb{R}^3}\mathcal{M}(\rho_{\e},U_{\e},T_{\e})\varphi \,dxdvdt.
\end{align*}
On the other hand, from the weak $L^1$ compactness of $\mathcal{M}(f_{\e})$, we can find a $L^1$ function $M$ such that
\begin{align*}
\lim_{\e\rightarrow0}\int_{E\times \mathbb{R}^3}\mathcal{M}(\rho_{\e},U_{\e},T_{\e})\varphi \,dxdvdt
=\int_{E\times \mathbb{R}^3} M \varphi \,dxdvdt.
\end{align*}
Thus we obtain
\begin{align*}
\int_{E\times \mathbb{R}^3}\mathcal{M}(\rho,U,T)\varphi \,dxdvdt\leq \int_{E\times \mathbb{R}^3} M \varphi \,dxdvdt,
\end{align*}
for all $\varphi \in L^{\infty}(\T^3 \times \R^3 \times (0,T_*))$, from which we can conclude that
\begin{align}\label{less than M}
\mathcal{M}(\rho,U,T)\leq M.
\end{align}
almost everywhere on $E\times \mathbb{R}^3$. Now, taking $\varphi=1$, we find
\begin{align*}
\int_{E\times \mathbb{R}^3}M\,dxdvdt
&=\lim_{\e\rightarrow0}\int_{E\times \mathbb{R}^3}\mathcal{M}(\rho_{\e},U_{\e},T_{\e})\,dxdvdt \cr &=\lim_{\e\rightarrow0}\int_{E}\rho_{\e}\,dxdt\cr
&=\int_{E}\rho \,dxdt\cr
& =\int_{E\times \mathbb{R}^3}\mathcal{M}(\rho,U,T)\,dxdvdt.
\end{align*}
This, together with (\ref{less than M}) implies $\mathcal{M}(\rho,U,T)= M$ almost everywhere on $E$. On the other hand, we observe
\begin{align*}
\lt|\lim_{\e\rightarrow 0}\int_{E^c\times \mathbb{R}^3}\mathcal{M}_\e\varphi \,dxdvdt\rt|
&\leq \lim_{\e\rightarrow 0}\int_{E^c\times \mathbb{R}^3}\mathcal{M}_\e|\varphi|\,dxdvdt\cr
&\leq \lim_{\e\rightarrow 0}\|\varphi\|_{L^{\infty}}\int_{E^c\times \mathbb{R}^3}\mathcal{M}_\e \,dxdvdt\cr
&= \lim_{\e\rightarrow 0}\|\varphi\|_{L^{\infty}}\int_{E^c}\rho_\e \,dxdt\cr
&=\|\varphi\|_{L^{\infty}}\int_{E^c}\rho \,dxdt\cr
&=0.
\end{align*}
Hence we obtain
\[
\lim_{\e\rightarrow 0}\int_{E^c\times \mathbb{R}^3}M\,dxdvdt=0.
\]
In conclusion, we have
\begin{align*}
\lim_{\e\rightarrow0}\int^{T_*}_0\!\!\int_{\mathbb{T}^3\times \mathbb{R}^3}\mathcal{M}(f_\e)\varphi \,dxdvdt
&=\lim_{\e\rightarrow0}\int_{E\times \mathbb{R}^3}\mathcal{M}(f_\e)\varphi \,dxdvdt
+\lim_{\e\rightarrow0}\int_{E^c\times \mathbb{R}^3}\mathcal{M}(f_\e)\varphi \,dxdvdt\cr
&=\int_{E\times \mathbb{R}^3}\mathcal{M}(f)\varphi \,dxdvdt
+0\cr
&=\int_0^{T_*}\!\!\int_{\mathbb{T}^3\times \mathbb{R}^3}\mathcal{M}(f)\varphi \,dxdvdt.
\end{align*}
This provides the desired result.

%
%
%
%

\subsection{Compactness of $u_\e$ in $L^2(0,T_*;L^2(\T^3))$}
In this subsection, we show that $u_\e$ is compact in $L^2(0,T_*;L^2(\T^3))$. 
For this, we are going to show that $\pa_t u_\e$ is uniformly bounded in $L^{3/2}(0,T_*;\mv')$.

It follows from the weak formulation for the fluid part that for all $\psi \in \mc^1(\T^3 \times [0,T_*])$ with $\nabla_x \cdot \psi = 0$ for almost everywhere $t$
$$\begin{aligned}
\int_0^t \int_{\T^3} \pa_t u_\e \cdot \psi \,dxds &= - \int_0^t \int_{\T^3} \lt( (\eta_\e \star u_\e) \cdot \nabla_x \rt)u_\e\cdot \psi\,dxds - \mu\int_0^t \int_{\T^3} \nabla_x u_\e : \nabla_x \psi\,dxds\cr
&\quad -\int_0^t \int_{\T^3 \times \R^3} f_\e (u_\e - v)\gamma_\e(v) \cdot \psi\,dxdvds\cr
&=: J_1 + J_2 + J_3.
\end{aligned}$$
Using the integration by parts together with the divergence free condition, we get
\[
J_1 = \int_0^t \int_{\T^3} \lt((\eta_\e \star u_\e) \cdot \nabla_x \rt) \psi \cdot u_\e \,dxds,
\]
and then we estimate it as
$$\begin{aligned}
\lt| J_1 \rt| &\leq \int_0^t \|\nabla_x \psi\|_{L^2} \| |\eta_\e \star u_\e | | u_\e|\|_{L^2}\,ds\cr
& \leq \int_0^t \|\nabla_x \psi\|_{L^2}\| u_\e\|_{L^4}^2\,ds \cr
&\leq \|\nabla_x \psi\|_{L^3(0,T_*; L^2)}\|u_\e\|_{L^3(0,T_*;L^4)}^2,
\end{aligned}$$
where $\|u_\e\|_{L^3(0,T_*;L^4)}$ is uniformly bounded in $\e$ since $u_\e$ is uniformly bounded in $L^\infty(0,T_*;L^2(\T^3))$ $\cap L^2(0,T_*;$ $H^1(\T^3))$ and  the Sobolev embedding 
\[
L^\infty(0,T_*;L^2(\T^3)) \cap L^2(0,T_*;H^1(\T^3)) \hookrightarrow L^3(0,T_*;L^4(\T^3)).
\] 
Thus we obtain
\[
\psi \mapsto- \int_0^t \int_{\T^3} \lt( (\eta_\e \star u_\e) \cdot \nabla_x \rt)u_\e\cdot \psi\,dxds
\]
is bounded in $L^{3/2}(0,T_*;\mv')$. The estimate of $J_2$ can be easily done as
\[
\lt| J_2 \rt| \leq \mu \int_0^t \|\nabla_x u_\e\|_{L^2} \|\nabla_x \psi\|_{L^2}\,ds \leq  \|\nabla_x \psi\|_{L^3(0,T_*; L^2)} \|\nabla_x u_\e\|_{L^{3/2}(0,T_*;L^2)}.
\]
Thus it gives the same result as the above. Finally, we estimate $J_3$ as
$$\begin{aligned}
\lt| J_3 \rt| &\leq \int_0^t \int_{\T^3 \times \R^3} f_\e \lt(|u_\e| + |v| \rt) |\psi|\,dxdvds \cr
&\leq \int_0^t \lt( \|u_\e\|_{L^6}\|\psi\|_{L^6}\|\rho_\e\|_{L^{3/2}} + \|\psi\|_{L^5}\|m_1 f_\e\|_{L^{5/4}}\rt)\,ds\cr
&\leq \|u_\e\|_{L^2(0,T_*;L^6)}\|\psi\|_{L^2(0,T_*; L^6)}\|\rho_\e\|_{L^\infty(0,T_*;L^{3/2})} + \|\psi\|_{L^2(0,T_*;L^5)}\|m_1 f_\e\|_{L^2(0,T_*;L^{5/4})}.
\end{aligned}$$
On the other hand, it follows from Lemma \ref{lem_mom}; $m_0f\leq C_{\|f_\e\|_{L^{\infty}}}(m_2f)^{3/5}$ and H\"older inequality that
\begin{align*}
\|\rho_{f_\e}\|_{L^{3/2}} &\leq C\lt(\int_{\T^3} (m_2 f_\e)^{9/10} dx \rt)^{2/3}\cr
&\leq \left(\int_{\T^3} \left\{(m_2f_\e)^{9/10}\right\}^{10/9}dx\right)^{9/10}\left(\int_{\T^3} 1^{10}dx\right)^{1/10}\cr
&\leq C\lt(M_2 f_\e\rt)^{3/5}.
\end{align*}
Thus we get the uniform boundedness of $\|\rho_{f_\e}\|_{L^\infty(0,T_*;L^{3/2})}$ in $\e$. Similarly, we find
\[
\|m_1 f_\e\|_{L^{5/4}} \leq C(M_2 f_\e)^{4/5},
\]
i.e., $m_1 f_\e$ is uniformly bounded in $L^2(0,T_*; L^{5/4}(\T^3))$. Combined with the uniform boundedness of $\|u_\e\|_{L^2(0,T_*;L^6)}$ in $\e$, this yields
\[
\lt| J_3 \rt| \leq C\|\psi\|_{L^2(0,T_*; L^6)} \leq C\|\psi\|_{L^2(0,T_*; H^1)}.
\]
Thus we obtain that $\pa_t u_\e$ is uniformly bounded in $L^{3/2}(0,T_*;\mv')$. Then, by Aubin-Lions lemma, we have the following strong convergences of $u_\e$:
$$\begin{aligned}
u_\e \to u \quad \mbox{in} \quad L^2(0,T_*;L^2(\T^3)), \quad u_\e \to u \quad \mbox{in} \quad \mc([0,T_*];\mv'),
\end{aligned}$$
as $\e \to 0$. These convergence together with the weak convergences allow us to pass to the limit to conclude the existence of weak solutions.

In order to extend that local-in-time weak solutions to the global ones, we give the following energy estimate showing the total energy of the system \eqref{main_sys} is not increasing. Then, by using the same strategy based on the continuity argument as in \cite[Section 3.6]{BDGM}, we have the global-in-time existence of weak solutions and complete the proof of Theorem \ref{thm_main}. Even though the proof of following lemma is almost same with \cite[Lemma 2]{BDGM}, for the completeness and the readers' convenience, we provide its details in Appendix \ref{app_lem_glo}.

\begin{lemma}\label{lem_glo}Let $(f,u)$ be the solutions to the system \eqref{main_sys} obtained above. Then we have the following total energy estimate
\begin{align*}
\begin{aligned}
&\frac12M_2f(t) + \frac12\|u(t)\|_{L^2}^2 + \mu\int_0^t\|\nabla u(s)\|_{L^2}^2 ds\cr
&\hspace{2cm} + \int_0^t \int_{\T^3 \times \R^3} f|u - v|^2 dxdv ds \leq \frac12M_2f_0 + \frac12\|u_0\|_{L^2}^2
\end{aligned}
\end{align*}
for almost every $t \in [0,T_*]$.
\end{lemma}

%
%
%
%

\section*{Acknowledgments}
Young-Pil Choi is supported by National Research Foundation of Korea(NRF) grants funded by the Korea government(MSIP) (No. 2017R1C1B2012918 and 2017R1A4A1014735) and POSCO Science Fellowship of POSCO TJ Park Foundation. Seok-Bae Yun is supported by Basic Science Research Program through NRF funded by the Ministry of Education(No. 2016R1D1A1B03935955).

%
%
%
%
\appendix

\section{Proof of Lemma \ref{lem_glo}}\label{app_lem_glo}

A straightforward computation yields
\begin{align*}
\begin{aligned}
&\frac12M_2f_\e(t) + \frac12\|u_\e(t)\|_{L^2}^2 + \mu\int_0^t\|\nabla u_\e(s)\|_{L^2}^2 ds\cr
&\hspace{1.2cm} + \int_0^t \int_{\T^3 \times \R^3} f_\e|u_\e - v|^2 \,dxdv ds = \frac12M_2f_0^\e + \frac12\|u_0^\e\|_{L^2}^2 + R_\e(t),
\end{aligned}
\end{align*}
where the remnant $R_\e(t)$ is given by
\begin{align*}
\begin{aligned}
R_\e(t) &= \int_0^t \int_{\T^3 \times \R^3} f_\e |u_\e|^2(1 - \gamma_\e(v))\,dx dv ds  - \int_0^t\int_{\T^3 \times \R^3} f_\e u_\e\cdot v(1 - \gamma_\e(v))\,dxdv ds\cr
& \quad + \int_0^t \int_{\T^3 \times \R^3} f_\e (u_\e - \eta_\e \star u_\e)\cdot v \,dx dv ds\cr
& =: R_\e^1 + R_\e^2 + R_\e^3.
\end{aligned}
\end{align*}
We now show $R_\e(t) \to 0$ as $\e \to 0$ uniformly in $t \in [0,T_*]$.

\noindent $\bullet$ {\bf Estimate of $R_\e^1(t)$:} Set $h_\e(x,t) := \int_{\R^3}f_\e(x,v,t)(1 - \gamma_\e(v))\,dv$. Then we use Lemma \ref{lem_mom} to obtain
\begin{align*}
\begin{aligned}
|R_\e^1(t)| &\leq \int_0^t \int_{\T^3} |u_\e|^2|h_\e|dxds \leq \int_0^t \|u_\e\|_{L^6}^2\|m_0h_\e\|_{L^{3/2}} ds\cr
&\leq C\int_0^t\|u_\e\|_{H^1}^2|M_{3/2}h_\e|^{3/2}ds \cr
&\leq C\|M_{3/2}h_\e\|_{L^\infty(0,T_*;L^{3/2})}^{3/2}\|u_\e\|_{L^2(0,T_*;H^1)}^2.
\end{aligned}
\end{align*}
On the other hand, we find
\begin{align*}
\begin{aligned}
|M_{3/2}h_\e(t)| &\leq \int_{\T^3 \times \R^3}|v|^{3/2}f_\e(1 - \gamma_\e) dx dv\cr
&\leq \int_{\T^3 \times \{v:|v| \geq \frac{1}{2\e}\}}|v|^{3/2}f_\e dx dv\cr
&\leq \sqrt{2\e}\int_{\T^3 \times \R^3}|v|^2f_\e dx dv \leq C\sqrt{\e}.
\end{aligned}
\end{align*}
Thus we have
\[
|R_\e^1(t)| \leq C\sqrt\e \to 0 \quad \mbox{as } \e \to 0.
\]
\newline
$\bullet$ {\bf Estimate of $R_\e^2(t)$:} Taking a similar argument as the above, we get
\begin{align*}
\begin{aligned}
|R_\e^2(t)| &\leq \int_0^t\int_{\T^3}|u_\e||m_1h_\e|dxds \leq \int_0^t \|u_\e\|_{L^6}\|m_1h_\e\|_{L^{6/5}} ds\cr
&\leq \int_0^t\|u_\e\|_{H^1}|M_{9/5}h_\e|^{5/6}ds\cr
&\leq \sqrt{T}\|u_\e\|_{L^2(0,T_*;H^1)}\|M_{9/5}h_\e\|_{L^\infty(0,T_*;L^{5/3})}^{5/6} \cr
&\leq C\e^{1/5} \to 0 \quad \mbox{as } \e \to 0,
\end{aligned}
\end{align*}
where we used
\[
|M_{9/5}h_\e(t)| \leq \int_{\T^3 \times \{v:|v|\geq \frac{1}{2\e}\}}|v|^{9/5}f_\e \,dx dv \leq (2\e)^{1/5}M_2f_\e(t) \leq C\e^{1/5}.
\]
\newline
$\bullet$ {\bf Estimate of $R_\e^3(t)$:} We again divide it into two terms $R_{\e,\delta}^{3,i},i=1,2$ as follows.
\begin{align*}
\begin{aligned}
R_\e^3(t) &= \int_0^t\int_{\T^3 \times \R^3} f_\e(u_\e - \eta_\e \star u_\e)\cdot v (1 - \gamma_\delta(v))\,dx dv ds\cr
& + \int_0^t\int_{\T^3 \times \R^3}f_\e(u_\e - \eta_\e \star u_\e) \cdot v \gamma_\delta(v)\,dx dv ds\cr
&=: R_{\e,\delta}^{3,1}(t) + R_{\e,\delta}^{3,2}(t),
\end{aligned}
\end{align*}
for any $\delta > 0$. First, we easily find that $|R_{\e,\delta}^{3,1}(t)| \leq C\delta^{1/5} \to 0$ as $\delta \to 0$  uniformly in $\e$ using the same argument as the above. For the estimate $R_{\e,\delta}^{3,2}$, we use the uniform bound estimate of $f_\e$ in $L^\infty(\T^3 \times\R^3 \times (0,T_*))$ to get
$$\begin{aligned}
|R_{\e,\delta}^{3,2}(t)| &\leq \int_0^t \int_{\T^3 \times \{v:|v|\leq \frac1\delta\}} |u_\e - \eta_\e\star u_\e||f_\e||v|dxdv ds  \leq C_\delta\|f_\e\|_{L^\infty}\|u_\e - \eta_\e\star u_\e\|_{L^1(0,T_*;L^1)}.
\end{aligned}$$
Then since $u_\e \to u$ in $L^2(0,T_*;L^2_{loc}(\T^3))$ we obtain
\[
|R_{\e,\delta}^{3,2}(t)| \to 0 \quad \mbox{as} \quad \e \to 0.
\]
Thus we first let $\e \to 0$ to have $|R_\e^1(t)| + |R_\e^2(t)| + |R_\e^3(t)|\leq C\delta^{1/5}$ for all $\delta>0$, and then let $\delta \to 0$ to have $R_\e \to 0$ as $\e \to 0$ uniformly in $t \in [0,T_*]$. We next use the weak$-\star$ convergence of $f_\e$ in $L^\infty(\T^3 \times \R^3 \times (0,T_*))$ to get
\[
M_2f(t) \leq \liminf_{\e \to 0}M_2 f_\e(t) \quad \mbox{for almost every } t \in [0,T_*].
\]
Using that idea together with the strong convergence of $u_\e$ in $L^2(\T^3 \times (0,T_*))$, we can also deal with the terms $\int_0^t\int_{\T^3 \times \R^3}f_\e|u_\e - v|^2\,dxdvds$, $\|u_\e\|_{L^2(\T^3)}$, and $\int_0^t \|\nabla_x u_\e(s)\|_{L^2}^2\,ds$. This completes the proof.
%
%
%
%

\end{document}